\documentclass[final]{siamltex}
\usepackage{graphicx}
\usepackage{amsmath,amsfonts}
\usepackage{color}
\usepackage{bm}

\newtheorem{thm}{Theorem}
\newtheorem{rmk}[thm]{Remark}
\newcommand{\ignore}[1]{}
\newcommand{\half}{\frac{1}{2}}

\title{Efficient and Long-Time Accurate Second-Order Methods for Stokes-Darcy Systems}

\author {Wenbin Chen\thanks{School of Mathematical Sciences, Fudan University ({\tt wbchen@fudan.edu.cn}).} \and Max Gunzburger\thanks{Department of Scientific Computing, Florida State University ({\tt gunzburg@fsu.edu)}.} \and Dong Sun\thanks{Department of Mathematics,
Florida State University ({\tt dsun@math.fsu.edu}).} \and Xiaoming Wang\thanks{Department of Mathematics,
Florida State University ({\tt wxm@math.fsu.edu}).}}

\newcommand{\vm}[1]{\vec{\mathbf{#1}}}

\def\btau{{\boldsymbol\tau}}

\begin{document}
\maketitle

\begin{abstract}
We propose and study two second-order in time implicit-explicit (IMEX) methods for the coupled Stokes-Darcy system that governs flows in karst aquifers. The first is a combination of a second-order backward differentiation formula and the second-order Gear's extrapolation approach. The second is a combination of the second-order Adams-Moulton and second-order Adams-Bashforth methods. Both algorithms only require the solution of two decoupled problems at each time step, one Stokes and the other Darcy. Hence, these schemes are very efficient and can be easily implemented using legacy codes. We establish the unconditional and uniform in time stability for both schemes. The uniform in time stability leads to uniform in time control of the error which is highly desirable for modeling physical processes, e.g., contaminant sequestration and release, that occur over very long time scales. Error estimates for fully-discretized schemes using finite element spatial discretizations are derived. Numerical examples are provided that illustrate the accuracy, efficiency, and long-time stability of the two schemes.
\end{abstract}



\pagestyle{myheadings}
\thispagestyle{plain}
\markboth{W. Chen, M. Gunzburger, D. Sun, and X. Wang}{Long-time accurate 2nd-order methods for Stokes-Darcy systems}

\section{Introduction}\label{sec:1}

Karst is a common type of landscape formed by the dissolution of layers of soluble bedrock, usually including carbonate rock, limestone, and dolomite. Karst regions often contain karst aquifers which are important sources of potable water. For example, about 90\% of the fresh water used in Florida comes from karst aquifers \cite{kin}. Clearly, the study of karst aquifers is of great importance, especially because they are seriously threatened by contamination \cite{kun}.

A karst aquifer, in addition to a porous limestone or dolomite matrix, typically has large cavernous conduits that are known to have great impact on groundwater flow and contaminant transport within the aquifer. During high-rain seasons, the water pressure in the conduits is larger than that in the ambient matrix so that conduit-borne contaminants can be driven into the matrix. During dry seasons, the pressure differential reverses and contaminants long sequestered in the matrix can be released into the
free flow in the conduits and exit through, e.g., springs and wells, into surface water systems. Therefore, the understanding of the interaction between the free flow in the conduits and the Darcy flow in the matrix is crucial to the study of groundwater flows and contaminant transport in karst region.

The mathematical study of flows in  karst aquifers is  a well-known challenge due to the coupling of the flow in the conduits and the flow in the surrounding matrix, the complex geometry of the network of conduits, the vastly disparate spatial and temporal scales, the strong heterogeneity of the physical parameters, and the huge associated uncertainties in the data. Even for a small, lab-size conceptual model with only one conduit (pipe) imbedded in a homogenous porous media (matrix), significant mathematically rigorous progress has only recently been achieved. For the coupled Stokes-Darcy model that includes the classical Beavers-Joseph \cite{Beavers1967} matrix-conduit interface boundary condition, see \cite{Cao2010, Cao2010b, Cao2011}. For various simplified interface conditions, see, e.g., \cite{Cao2010, Disc2003, Layton2003}. Nonlinear interface conditions have also been proposed for Navier-Stokes/Darcy modeling; see, e.g., \cite{Ces2012, Disc2009}.

Due to the practical importance of the problem of flow and contaminant transport in karst aquifers, there has been a lot of attention recently paid to the development of accurate and efficient numerical methods for the coupled Stokes-Darcy system; see, e.g., \cite{Cao2012, Disc2003, Layton2003, Mu2007} among many others. The efficiency of the algorithms is a particularly important issue due to the large scale of field applications. Because of the disparity of governing equations and physics in the conduit and matrix, domain decomposition methods (also called partitioned methods by some authors) that only requires separate Stokes and Darcy solves seems natural; see, e.g., \cite{Chen2010, Chen2011, Disc2003, Disc2007, Kubacki2012, Layton2011a, Layton2011b, Layton2012, Mu2010}. On the other hand, long-time accuracy of the schemes is also highly desirable because the physical phenomena of retention and release of contaminants takes place over a very long time scale. Therefore, there is a need to ensure the {\em long-time accuracy} of the discretization algorithms in addition to the standard notion of accuracy on an order one time scale.

The purpose of this work is to propose and investigate two types of numerical methods for the coupled Stokes-Darcy system. We discretize the system in time via either a combination of second-order BDF and and Gear extrapolation methods or a combination of second-order Adams-Moulton and Adams-Bashforth methods. These algorithms are special cases of the implicit-explicit (IMEX) class of schemes. The coupling terms in the interface conditions are treated explicitly in our algorithm so that only two decoupled problems (one Stokes and one Darcy) are solved at each time step. Therefore, these schemes can be implemented very efficiently and, in particular, legacy codes for each of the two components can be utilized. Moreover, we show that our schemes are unconditionally stable and long-time stable in the sense that the solutions remain uniformly bounded in time. The uniform in time bound of the solution further leads to uniform in time error estimates. This is a highly desirable feature because one would want to have reliable numerical results over the long-time scale of contaminant sequestration and release. Uniform in time error estimates for fully discrete schemes using finite element spatial discretizations are also presented. Our numerical experiments illustrate our analytical results.

Our work can be viewed as a time-dependent non-iterative version of the steady-state domain decomposition work in \cite{Chen2011, Disc2003} and as a generalization of the first-order schemes in \cite{Cao2012,Mu2010} that achieve the desirable second-order accuracy withtout increasing the complexity. The backward differentiation-based algorithm can be viewed as an infinite-dimensional version of the scheme presented in \cite{Layton2012}, but with the additional important result on time-uniform error estimates. The Adams-Moulton/Bashford based algorithm is new so far as the Stokes-Darcy problem is concerned. To the best of our knowledge, our uniform in time error estimates are the first of their kind for Stokes-Darcy and related systems.

The rest of the paper is organized as follows. In \S\ref{sec:2}, we introduce the coupled Stokes-Darcy system and the associated weak formulation as well as the two second-order in time schemes. The unconditional and long-time stability with respect to the $L^2$ norm are presented in \S\ref{sec:3}. Section \ref{sec:4} is devoted to the stability with respect to the $H^1$ norm. The $H^1$ estimates are important for the finite element analysis; this is another new feature of our work, even for first-order schemes. In \S\ref{sec:5}, we focus on the error analysis of the fully discretized scheme using finite element spatial discretizations. Numerical results that illustrate the accuracy, efficiency, and long-time stability of our our algorithms are given in \S\ref{sec:6}. We close by providing some concluding remarks in \S\ref{sec:7}.

\section{The Stokes-Darcy system and two types of IMEX methods}\label{sec:2}

\subsection{The Stokes-Darcy system}\label{sec:21}

For simplicity, we consider a conceptual domain for a karst aquifer that consists of a porous media (matrix), denoted by $\Omega_p\in{\mathbb R}^d$, and a conduit, denoted by $\Omega_f\in{\mathbb R}^d$, where $d=2,3$ denotes the spatial dimension. The interface between the matrix and the conduit is denoted $\Gamma$. The remaining parts of the boundaries of $\Omega_p$ and $\Omega_f$ are denoted by $\partial\Omega_p$ and  $\partial\Omega_f$, respectively. See Fig.~\ref{fig:1}.

\begin{figure}[h!]
\centering
\includegraphics[scale=0.4]{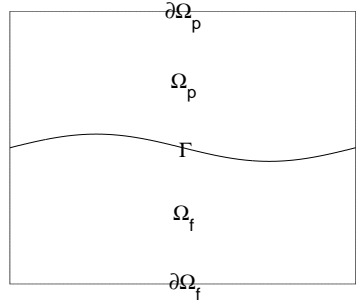}
\caption{The physical domain consisting of a porous media $\Omega_p$ and a free-flow conduit $\Omega_f$.}\label{fig:1}
\end{figure}

The {\em coupled Stokes-Darcy system} governing fluid flow in the karst system is given by \cite{Cao2010, Disc2003}
\begin{equation}\label{SDsys}
\left\{\begin{array}{ll}
\displaystyle S\frac{\partial \phi}{\partial t}-\nabla\cdot({{\mathbb K}\nabla \phi})=f & \quad \mbox{in $\Omega_p$} \\[1ex]
\displaystyle \frac{\partial \mathbf{u}_f}{\partial t}-\nabla\cdot{\mathbb T}\left(\mathbf{u}_f,p\right)=\mathbf{f}\quad\mbox{and}\quad
\nabla\cdot{\mathbf{u}_f}=0    & \quad \mbox{in $\Omega_f$},
\end{array}\right.
\end{equation}
where the unknowns are the fluid velocity $\mathbf{u}_f$ and the kinematic pressure $p$ in the conduit and the hydraulic head $\phi$ in the matrix; the velocity in the matrix is recovered from $\mathbf{u}_p=-{\mathbb K}\nabla \phi$. In \eqref{SDsys}, $\mathbf{f}$ and $f$ denote external body forces acting on the domains $\Omega_f$ and $\Omega_p$ respectively, and ${\mathbb T}(\mathbf{v},p) =  \nu (\nabla \mathbf{v} + \nabla^T \mathbf{v})-p{\mathbb I}$ denotes the stress tensor. The parameters appearing in \eqref{SDsys} are the water storage coefficient $S$, the hydraulic conductivity tensor ${\mathbb K}$, and the kinematic viscosity of the fluid $\nu$.

For simplicity, we assume homogeneous Dirichlet boundary conditions for the hydraulic head and fluid velocity on the outer boundaries $\partial\Omega_p$ and  $\partial\Omega_f$, respectively. On the interface $\Gamma$, we impose the {\em Beavers-Joseph-Saffman-Jones interface conditions} \cite{Beavers1967, Jones1973, Saff1971}
\begin{equation}\label{BJSJ}
\left\{\begin{aligned}
&\mathbf{u}_f\cdot{\bf n}_f=\mathbf{u}_p\cdot{\bf n}_f=-({\mathbb K}\nabla \phi)\cdot {\bf n}_f\\
&-{\btau}_j\cdot({\mathbb T}(\mathbf{u}_f,p_f)\cdot{{\bf n}_f})=\alpha_{BJSJ}{\btau}_j\cdot{\mathbf{u}_f},\quad j=1,\ldots,d-1\\
&-{\bf n}_f\cdot({\mathbb T}(\mathbf{u}_f,p_f)\cdot{{\bf n}_f})=g\phi,
\end{aligned}\right.
\end{equation}
where ${\bf n}_f$ denotes the outer unit normal vector for $\Omega_f$ and $\{{\btau}_j\}_{j=1,2,\ldots,d-1}$, denotes a linearly independent set of vectors tangent to the interface $\Gamma$. The additional parameters appearing in \eqref{BJSJ} are the gravitational constant $g$ and the Beavers-Joseph-Saffman-Jones coefficient $\alpha_{BJSJ}$.


\subsection{Weak formulation}\label{sec:22}

We denote by $(\cdot,\cdot)_{D}$ and $\|\cdot\|_D$ the standard $L^2(D)$ inner product and norm, respectively, where $D$ may be $\Omega_f$, $\Omega_p$, or $\Gamma$. We often suppress the subscript $D$ if there is no possibility of confusion. We define the spaces
$$
\begin{aligned}
 \mathbf{H}_f&=\left\{\mathbf{v}\in\left(H^1(\Omega_f)\right)^d\,\,\,\mid\,\,\, \mbox{$\mathbf{v}=\mathbf{0}$ on $\partial\Omega_f\setminus\Gamma$} \right\}\\
 {H}_p&=\left\{\psi\in H^1(\Omega_p)\,\,\,\mid\,\,\,\mbox{$\psi=0$ on $ \partial\Omega_p\setminus \Gamma$} \right\}\\
 {Q}&=L^2(\Omega_f), \qquad  \mathbf{W}=\mathbf{H}_f\times{H}_p.
\end{aligned}
$$
Dual spaces are denoted by $(\cdot)^\prime$ and duality parings between spaces and their duals induced by the $L^2$ inner product on the appropriate domain are denoted by $\langle \cdot\,,\cdot \rangle$.

A weak formulation of the Stokes-Darcy system \eqref{SDsys} is derived by multiplying the three equations in that system by test functions $\mathbf{v}\in \mathbf{H}_f$, $g\psi\in \mathbf{H}_p$, and $q\in {Q}$, respectively, then integrating over the corresponding domains, then integrating by parts the terms involving second-derivative operators, and then substituting the interface conditions \eqref{BJSJ} in the appropriate terms. The resulting weak formulation is given as follows \cite{Cao2010, Disc2002}: given $f\in (H_p)^\prime$ and ${\bf F}\in ({\bf H}_f)^\prime$, seek $\phi\in H_p$, ${\bf u}_f\in {\bf H}_f$, and  $p\in Q$, with $\partial\phi/\partial t\in (H_p)^\prime$ and $\partial{\bf u}/{\partial t}\in ({\bf H}_f)^\prime$, satisfying
\begin{equation}\label{weakform}
\begin{aligned}
\langle\langle\vm{u}_t,\vec{\mathbf{v}}\rangle\rangle+a(\vm{u}, \vm{v})+b(\mathbf{v},p)+a_\Gamma(\vm{u},\vec{\mathbf{v}})&=\langle\langle\langle\vec{\mathbf{f}},\vec{\mathbf{v}}\rangle\rangle\rangle \\  b(\mathbf{u},q)&=0,
\end{aligned}
\end{equation}
where $\vm{u}=[\mathbf{u},\phi]^T$, $\vec{\mathbf{v}}=[\mathbf{v},\psi]^T$, and $\vec{\mathbf{f}}=[\mathbf{f},gf]^T$ and where $(\cdot)_t=\partial(\cdot)/\partial t$. In \eqref{weakform}, we have that
\begin{equation}\label{bilinearforms}
\begin{aligned}
& \langle\langle\vm{u}_t,\vec{\mathbf{v}}\rangle\rangle=\langle\mathbf{u}_t, \mathbf{v}\rangle_{\Omega_f} + gS\langle\phi_t, \psi\rangle_{\Omega_p},\qquad
 b(\mathbf{v},q)=-(q, \nabla\cdot \mathbf{v})_{\Omega_f}
\\
&a(\vm{u}, \vm{v})= a_f(\mathbf{u},\mathbf{v})+a_p(\phi,\psi)+a_{BJSJ}(\mathbf{u},\mathbf{v})\\
& a_{\Gamma}(\vm{u},\vec{\mathbf{v}})=g(\phi, \mathbf{v}\cdot{\bf n}_f)_{\Gamma}-g(\mathbf{u}\cdot{\bf n}_f, \psi)_\Gamma,
\quad
\langle\langle\langle\vec{\mathbf{f}},\vec{\mathbf{v}}\rangle\rangle\rangle = \langle{\mathbf{f}},{\mathbf{v}}\rangle_{\Omega_f} +g\langle f ,\psi  \rangle_{\Omega_p},
\end{aligned}
\end{equation}
where
$$
\begin{aligned}
 a_f(\mathbf{u},\mathbf{v})=\nu(&\nabla\mathbf{u}, \nabla\mathbf{v})_{\Omega_f},\qquad
 a_p(\phi,\psi)=g({\mathbb K}\nabla\phi, \nabla\psi)_{\Omega_p}
\\&
 a_{BJSJ}(\mathbf{u},\mathbf{v})=\alpha_{BJ}(\mathbf{u}\cdot\vec{\tau}
,\mathbf{v}\cdot\vec{\tau})_{\Gamma}.
\end{aligned}
$$
In \eqref{weakform}, $\mathbf{u}_f$, $\phi$, and $p$ are the primary variables; as mentioned before, once the hydraulic head $\phi$ is known, one can recover $\mathbf{u}_p$, the velocity in the porous media, via the Darcy relation $\mathbf{u}_p=-{\bf K}\nabla\phi$.

It can be shown that the bilinear form $a(\cdot,\cdot)$ is {\em coercive}; indeed, we have that \cite{Cao2010}
\begin{equation}\label{ineq2}
  a({\vm{u}},{\vm{u}}) \ge (\nu\|\nabla
\mathbf{u}\|^2+gK_{\min}\|\nabla\phi\|^2  + \alpha_{BJ} \|\mathbf{u}\cdot\vec{\tau}\|^2_\Gamma)\ge C_a  \|\nabla \vm{u}\|^2,
\end{equation}
where $C_a=\min(\nu,gK_{\min})>0$ and where $K_{\min}$ denotes the smallest eigenvalue of ${\mathbb K}$. We define the norms $\|\vm{u}\|_a= a({\vm{u}},{\vm{u}})^\half$ and $\|\vm{v}\|_S=\langle\langle\vm{v},\vm{v}\rangle\rangle^\half$. We have that $\|\vm{v}\|_S$ is equivalent to the $L^2$ norm, i.e., we have that
\begin{equation}\label{eq:normS}
   c_s\|\vm{v}\|_S\le\|\vm{v}\|\le C_S \|\vm{v}\|_S.
 \end{equation}


\subsection{The second-order backward-differentiation scheme (BDF2)}\label{sec:23}

The first scheme we introduce discretizes in time via a second-order BDF whereas the interface term is treated via a second-order explicit Gear's extrapolation formula. We propose the following algorithm: for any $\vec{\mathbf{v}}\in \mathbf{W}$ and $q\in Q$,
\begin{equation}\label{schemeBDF}
\begin{aligned}
&\Big\langle\Big\langle\frac{3\vm{u}^{n+1}-4\vm{u}^n+\vm{u}^{n-1}}{2\Delta
t},\vec{\mathbf{v}}\Big\rangle\Big\rangle +a({\vm{u}}^{n+1},\vm{v}) +b(\mathbf{v},p^{n+1}) +a_{st}(\vm{u}^{n+1},\vec{\mathbf{v}})\\
&\qquad\qquad =\langle\langle\langle\vec{\mathbf{f}}^{n+1},{\vec{\mathbf{v}}}\rangle\rangle\rangle- \widetilde{a}_\Gamma(2\vm{u}^{n}-\vm{u}^{n-1},\vec{\mathbf{v}})\\
&b(\mathbf{u}^{n+1},q)=0,
\end{aligned}
\end{equation}
where the artificial stabilizing term $a_{st}(\cdot,\cdot)$ is defined as
\begin{equation}\label{ast}
a_{st}(\vm{u},\vec{\mathbf{v}})=\gamma_f
(\mathbf{u}\cdot{\bf n}_f, \mathbf{v}\cdot{\bf n}_f)_\Gamma+\gamma_p(\phi, \psi)_\Gamma
\end{equation}
with parameters $\gamma_f,\gamma_p\ge0$ and $\widetilde{a}_{\Gamma}(\vm{u},\vm{v})$ is defined as
$$
\widetilde{a}_{\Gamma}(\vm{u},\vm{v}) = a_\Gamma(\vm{u},\vm{v}) -
 a_{st}(\vm{u},\vm{v}).
$$

\subsection{The second-order Adams-Moulton-Bashforth method (AMB2)}\label{sec:24}

For the second scheme, we combine the second-order implicit Adams-Moulton treatment of the symmetric terms and the second-order explicit Adams-Bashforth treatment of the interface term to propose the following second-order scheme: for any $\vec{\mathbf{v}}\in \mathbf{W}$ and $q\in Q$,
\begin{equation}\label{schemeAMB}
\begin{aligned}
&\Big\langle\Big\langle\frac{\vm{u}^{n+1}-\vm{u}^n}{\Delta
t},\vec{\mathbf{v}}\Big\rangle\Big\rangle+a\left(D_\alpha{\vm{u}}^{n+1},\vm{v}\right)+b\left(\mathbf{v},D_\alpha{p}^{n+1}\right)
 +a_{st}\left(D_\alpha{\vm{u}}^{n+1},\vec{\mathbf{v}}\right)\\
& \qquad \qquad =\langle\langle\langle\vec{\mathbf{f}}^{n+\frac{1}{2}},{\vec{\mathbf{v}}}\rangle\rangle\rangle-
\widetilde{a}_\Gamma\Big(\frac{3}{2}\vm{u}^{n}-\frac{1}{2}\vm{u}^{n-1},\vec{\mathbf{v}}\Big)\\
&b\left(D_\alpha{\mathbf{u}}^{n+1},q\right)=0,
\end{aligned}\end{equation}
where $D_\alpha$ denotes the difference operator that  depends on a parameter $\alpha$ and is defined by $D_\alpha v^{n+1}=\alpha{v}^{n+1}+\left(\frac{3}{2}-2\alpha\right){v}^{n}+
\left(\alpha-\frac{1}{2}\right){v}^{n-1}$.
The stabilizing term $a_{st}(\cdot,\cdot)$ is defined as in \eqref{ast}.

\subsection{Efficiency of the schemes}\label{sec:25}

The implemented schemes are highly efficient because we can decouple the Stokes and Darcy subproblems:

\begin{enumerate}
\item given $\vm{u}^{n},\vm{u}^{n-1}$
\item set $\vec{\mathbf{v}}=[\mathbf{v},0]$ so that all terms involving $\phi$ drop out and we only need to use a fast Stokes solver to obtain $\mathbf{u}^{n+1}$;
\item set $\vec{\mathbf{v}}=[0,\psi]$ so that all terms involving $\mathbf{u}$ drop out and we only need a fast Poisson solver to obtain $\phi^{n+1}$;
\item Set $n=n+1$ and return to step 1.
\end{enumerate}
Note that steps 2 and 3 can be solved independently. Moreover, legacy codes can be used in each of those steps.

\section{Unconditional and long-time stability}\label{sec:3}

The goal of this section is to demonstrate the unconditional and long-time stability, with respect to the $L^2$ norm, of the two second-order schemes proposed in \S\ref{sec:2}. We first recall a few basic facts and notations that are needed below.

Recall that the $G$-matrix associated with the classical second-order BDF is given by
$$
G=\left(\begin{array}{cc}\frac{1}{2} & -1 \\ -1 &
\frac{5}{2}
\end{array} \right)
$$
with the associated  $G$-norm given by
$\|\mathbf{w}\|^2_{G}=\big(\mathbf{w}, G\mathbf{w}\big)
\quad \forall\, \mathbf{w} \in (L^2(\Omega))^2 .$
The following identity is well-known  (see, e.g., \cite{Hairer2002}): for any $v_i\in L^2(\Omega)$, $i=0,1,2$,
\begin{equation}\label{Gnormid}
\Big(\frac{3}{2}v_2-2v_1+\frac{1}{2}v_0,v_2\Big)
 =\frac{1}{2}\left(\|\mathbf{w}_1\|^2_{G}-\|\mathbf{w}_0\|^2_{G}\right)
 +\frac{\|v_2-2v_1+v_0\|^2}{4},
\end{equation}
where $\mathbf{w}_0=[v_0,v_1]^T$ and $\mathbf{w}_1=[v_1,v_2]^T$. We also apply the $G$ matrix to functions belonging to $\mathbf{W}$: for any $\mathbf{w}\in \mathbf{W}^2$, define
$|\mathbf{w}|_G^2=\langle \mathbf{w}, G\mathbf{w}\rangle.$ Then, for any $\vm{v}_i\in W$, $i=0,1,2$,
$$
 \Big\langle\Big\langle\frac{3}{2}\vm{v}_2-2\vm{v}_1+\frac{1}{2}\vm{v}_0,\vm{v}_2  \Big\rangle\Big\rangle
 =\frac{1}{2}\left(|\mathbf{w}_1|^2_{G}- |\mathbf{w}_0|^2_{G}\right)
 +\frac{\|\vm{v}_2-2\vm{v}_1+\vm{v}_0\|_S^2}{4},
$$
where $\mathbf{w}_0=[\vm{v}^0,\vm{v}^1]^T$ and $\mathbf{w}_1=[\vm{v}_1,\vm{v}_2]^T$.

The $G$-norms are equivalent norms on $(L^2(\Omega))^2$ in the sense that there exists $C_l, C_u >0$ such that
$$
C_l\|\mathbf{w}\|^2_{G}\le\|\mathbf{w}\|^2\le C_u\|\mathbf{w}\|^2_{G}\qquad\mbox{and} \qquad C_l\|\mathbf{w}\|^2_{G}\le|\mathbf{w}|^2_{G}\le C_u\|\mathbf{w}\|^2_{G}.
$$

We next recall the following basic inequalities:
\begin{itemize}
 \item trace inequality: if $\vec{\mathbf{v}}\in \mathbf{W}$, then
\begin{equation}\label{Inq:trace}
       \|\vec{\mathbf{v}}\|_{\Gamma}\le C_{tr} \sqrt{\|\vec{\mathbf{v}}\|\|\nabla\vec{\mathbf{v}}\|} , \qquad
       \|\vec{\mathbf{v}}\|_{\Gamma}\le C_{tr} \|\nabla\vec{\mathbf{v}}\|
\end{equation}
\item Poincar\'{e} inequality:\quad if $\vec{\mathbf{v}}\in \mathbf{W}$, then
$
\|\vec{\mathbf{v}}\|\le C_{P} \|\nabla\vec{\mathbf{v}}\|
$
 \item Young inequality:\quad
$
  a^{\frac{1}{2}}b^{\frac{1}{2}}c\le \frac{\varepsilon a^2}{4}+\frac{ b^2}{4\varepsilon^3 }+\frac{\varepsilon c^2}{2} \quad \forall\, a, b, c, \varepsilon >0 .
$
\end{itemize}
Other variants of Young's inequality will also be used.

The following estimate follows from the basic inequalities.

\begin{lemma}\label{Lemma1}
Let $a_{\gamma}(\cdot,\cdot)$ and $a_{st}(\cdot,\cdot)$ be defined as in \eqref{bilinearforms} and \eqref{ast}, respectively. Then, there exists a constant $C_{ct}$ such that
$$
|a_{st}(\vm{u},\vec{\mathbf{v}})|+|a_\Gamma(\vm{u},\vec{\mathbf{v}})|\le C_{ct}\|\vm{u}\|_\Gamma\|\vec{\mathbf{v}}\|_\Gamma
\qquad\forall\,\vm{u}, \vec{\mathbf{v}}\in \mathbf{W}.
$$
\end{lemma}

\begin{proof}
By the definition \eqref{ast} of $a_{st}(\cdot,\cdot)$, we have
\begin{equation}\label{astineq}
\begin{aligned}
|a_{st}(\vm{u},\vec{\mathbf{v}})|
&\le \gamma_f |( \mathbf{u}\cdot {\bf n}_f, \mathbf{v}\cdot {\bf n}_f)_\Gamma|
+\gamma_p|( \phi, \psi)_{\Gamma}|\\
&\le \gamma_f \|\mathbf{u}\cdot {\bf n}_f\|_\Gamma
\|\mathbf{v}\cdot
\vec{n}_f\|_\Gamma+\gamma_p\|\phi\|_\Gamma\|\psi\|_\Gamma\\
&\le \gamma_{\max}\left(\|\mathbf{u}\cdot {\bf n}_f\|_\Gamma
\|\mathbf{v}\cdot
{\bf n}_f\|_\Gamma + \|\phi\|_\Gamma\|\psi\|_\Gamma\right),
\end{aligned}
\end{equation}
where the triangle and Cauchy-Schwarz inequalities are used and
$\gamma_{\max}=\max\{\gamma_f,\gamma_p\}$. Similarly, by the definition \eqref{bilinearforms} of $a_{\Gamma}(\cdot,\cdot)$, we have
\begin{equation}\label{agamma}
|a_{\Gamma}(\vm{u},\vec{\mathbf{v}})|
\le g\left( \|\phi\|_\Gamma \|\mathbf{v}\cdot {\bf n}_f\|_\Gamma + \|\mathbf{u}\cdot {\bf n}_f\|_\Gamma\|\psi\|_\Gamma\right).
\end{equation}
Note that $\vm{u}=[\mathbf{u}, \phi]^T$ so that
$$
\|\vm{u}\|^2_\Gamma=\|\mathbf{u}\|_\Gamma^2 + \|\phi\|_\Gamma^2=\|\mathbf{u}\cdot{\bf n}_f\|_\Gamma^2 + \|\mathbf{u}\cdot{\boldsymbol\tau}\|^2_\Gamma+\|\phi\|^2_\Gamma.
$$
Then, combining (\ref{astineq}) and (\ref{agamma}), we obtain
$$
\begin{aligned}
|a_{st}(\vm{u}&,\vec{\mathbf{v}})|+|a_\Gamma(\vm{u},\vec{\mathbf{v}})|
\\    &\le (\gamma_{\max}(\|\mathbf{u}\cdot\vec{n}_f\|_\Gamma \|\mathbf{v}\cdot
\vec{n}_f\|_\Gamma + \|\phi\|_\Gamma \|\psi\|_\Gamma )+g(
\|\phi\|_\Gamma\|\mathbf{v}\cdot \vec{n}_f\|_\Gamma + \|\mathbf{u}\cdot
\vec{n}_f\|_\Gamma \|\psi\|_\Gamma))\\
  &\le \max\{\gamma_{\max},g\} (\|\mathbf{u}\cdot
\vec{n}_f\|_\Gamma  +\|\phi\|_\Gamma  )( \|\mathbf{v}\cdot \vec{n}_f\|_\Gamma+\|\psi\|_\Gamma)\\
&\le  \sqrt{2}\max\{\gamma_{\max},g\} \|\vm{u}\|_{\Gamma}\|\vec{\mathbf{v}}\|_{\Gamma}.
\end{aligned}
$$
The lemma is proved by setting $C_{ct}=\sqrt{2}\max\{\gamma_{\max},g\}$.
\end{proof}

For the sake of brevity, we introduce the {\em BDF difference operator}
$Dv^{n+1}= \frac{3}{2}v^{n+1}-2v^n+\half v^{n-1}$ and the {\em central difference operator} $\delta v^{n+1}= v^{n+1}-2v^{n}+v^{n-1}$.

\subsection{Unconditional stability of the the BDF2 and AMB2 schemes}\label{sec:31}

\subsubsection{Unconditional stability of the BDF2 scheme}\label{sec:311}
\textcolor{white}{xxx}

\begin{theorem}\label{thm31}
Let $T>0$ be any fixed time. Then, the {\em BDF2} scheme \eqref{schemeBDF} is unconditionally stable on $(0,T]$.
\end{theorem}
\begin{proof}
Setting
$\vec{\mathbf{v}}=\vm{u}^{n+1}=\left(\mathbf{u}^{n+1},\phi^{n+1}\right)$ in the BDF2 scheme (\ref{schemeBDF}), we have
$$
\begin{aligned}
 \frac{1}{\Delta t}\big\langle\big\langle D\vm{u}^{n+1},\vm{u}^{n+1}\big\rangle\big\rangle&+a({\vm{u}}^{n+1},{\vm{u}}^{n+1})
+a_{st}(\delta \vm{u}^{n+1},\vm{u}^{n+1})\\
&=\big\langle\big\langle\big\langle\vec{\mathbf{f}}^{n+1},\vec{{\mathbf{u}}}^{n+1}\big\rangle\big\rangle\big\rangle-a_\Gamma(2\vm{u}^{n}-\vm{u}^{n-1},\vm{u}^{n+1}).
\end{aligned}
$$
From \eqref{Gnormid} and the skew-symmetry of $a_{\Gamma}(\cdot,\cdot)$, we obtain
\begin{equation}\label{ineq1}
\begin{aligned}
\half|\vec{\mathbf{w}}_n|^2_G-&\half|\vec{\mathbf{w}}_{n-1}|^2_G+  \frac{1}{4}{\|\delta \vm{u}^{n+1}\|_S^2} +\Delta
t a({\vm{u}}^{n+1},{\vm{u}}^{n+1}) +\Delta
ta_{st}(\vm{u}^{n+1},\vm{u}^{n+1})\\
& =\Delta t\Big(
\big\langle\big\langle\big\langle\vec{\mathbf{f}}^{n+1},\vec{{\mathbf{u}}}^{n+1}\big\rangle\big\rangle\big\rangle+
\widetilde{a}_\Gamma(-2\vm{u}^{n}+\vm{u}^{n-1},\vm{u}^{n+1})\Big),
\end{aligned}
\end{equation}
where $\vec{\mathbf{w}}_{n}=[\vm{u}^{n+1}, \vm{u}^n]^T$. Also, from the definition of the bilinear form $\widetilde{a}_{st}(\cdot,\cdot)$, Lemma \ref{Lemma1}, the trace inequality, and Young's inequality, we have
\begin{equation}\label{BDF:interface1}
\begin{aligned}
\widetilde{a}_{\Gamma}(-2\vm{u}^{n}&+\vm{u}^{n-1},\vm{u}^{n+1})
  \le C_{ct}  \|-2\vm{u}^{n}+\vm{u}^{n-1}\|_\Gamma\|\vm{u}^{n+1}\|_\Gamma \\
&\le C_{ct}C^2_{tr} \|-2\vm{u}^{n}+\vm{u}^{n-1}\|^{\half}\|-2\nabla\vm{u}^{n}+\nabla\vm{u}^{n-1}\|^{\half}\|\nabla\vm{u}^{n+1}\|\\
&\le C_{ct}C^2_{tr} \|-2\vm{u}^{n}+\vm{u}^{n-1}\|^{\half}\|\nabla\vm{u}^{n+1}\|\left(\sqrt{2}\|\nabla\vm{u}^{n}\|^{\half}+\|\nabla\vm{u}^{n-1}\|^{\half}\right)\\
&\le\frac{C_1}{2} |\vec{\mathbf{w}}_{n-1}|^2+\frac{C_a }{6}\|\nabla\vm{u}^{n+1}\|^2+\frac{C_a }{3}\|\nabla\vm{u}^{n}\|^2\\
&\qquad+\frac{C_2}{2} |\vec{\mathbf{w}}_{n-1}|^2+\frac{C_a }{6}\|\nabla\vm{u}^{n+1}\|^2+\frac{C_a }{6}\|\nabla\vm{u}^{n-1}\|^2.
\end{aligned}
\end{equation}
For the forcing term, we have
\begin{equation}\label{forcingterm}
  \big\langle\big\langle\big\langle\vec{\mathbf{f}}^{n+1},\vec{{\mathbf{u}}}^{n+1}\big\rangle\big\rangle\big\rangle
 \le\frac{C_3}{2} \|\vec{\mathbf{f}}^{n+1}\|^2+\frac{C_a}{6 {C^2_P}}\left\|{\vm{u}^{n+1}}\right\|^2.
\end{equation}
After we discard the nonnegative terms $ \|\delta \vm{u}^{n+1}\|^2_S $ and $a_{st}(\vm{u}^{n+1},\vm{u}^{n+1})$, noting that $\|\nabla\vm{u}^{n+1}\|\ge\frac{1}{C_P}\|\vm{u}^{n+1}\|$, and using \eqref{BDF:interface1} and \eqref{forcingterm}, \eqref{ineq1} becomes
$$
\begin{aligned}
|\vec{\mathbf{w}}_n|^2_G&+C_a\Delta t \|\nabla \vm{u}^{n+1}\|^2\le C_3\|\vec{\mathbf{f}}^{n+1}\|^2\Delta t\\
& +(1+(C_1+C_2)\Delta t)|\vec{\mathbf{w}}_{n-1}|^2_G+\frac{2C_a\Delta t}{3}
\|\nabla \vm{u}^{n}\|^2+\frac{C_a\Delta t}{3} \|\nabla \vm{u}^{n-1}\|^2.
\end{aligned}
$$
Next, by adding $\frac{C_a\Delta t}{3}\|\nabla \vm{u}^{n}\|^2$ to both sides of this inequality, we deduce
$$
\begin{aligned}
&E_n+\frac{C_a\Delta t^2(C_1+C_2)}{(1+(C_1+C_2)\Delta t)} \|\nabla \vm{u}^{n+1}\|^2+\frac{C_a\Delta t^2(C_1+C_2)}{3(1+(C_1+C_2)\Delta t)} \|\nabla \vm{u}^{n}\|^2\\
&\qquad\le C_3\|\vec{\mathbf{f}}^{n+1}\|^2\Delta t + (1+(C_1+C_2)\Delta t)E_{n-1}\\
&\qquad\le e^{(C_1+C_2)T}E_0+\frac{C_3}{C_1+C_2}e^{(C_1+C_2)T}\max_n\|\vec{\mathbf{f}}^{n+1}\|^2,
\end{aligned}
$$
where
$$
E_{n}=|\vec{\mathbf{w}}_n|^2_G+\frac{C_a\Delta t}{(1+(C_1+C_2)\Delta t)} \|\nabla \vm{u}^{n+1}\|^2+
\frac{C_a\Delta t}{3(1+(C_1+C_2)\Delta t)} \|\nabla \vm{u}^{n}\|^2.
$$
Thus, the unconditional stability of the BDF2 scheme is proved.
\end{proof}

\subsubsection{Unconditional stability of the AMB2 scheme}\label{sec:312}

We introduce the parameters
\begin{equation}\label{alpha-beta}
\begin{aligned}
&\alpha_1=\Big|\frac{3}{2}-2\alpha\Big|, \qquad \alpha_2=\Big|\alpha-\half\Big|, \qquad
\\&
\beta_3= \alpha_1+ \alpha_2, \qquad \beta_1=2\alpha- \beta_3, \qquad \beta_2=\frac12(\beta_1+\beta_3).
\end{aligned}
\end{equation}

\begin{theorem}\label{theorem32}
Let $T>0$ be any fixed time and let $1/2<\alpha<1$. Then, the {\em AMB2} scheme \eqref{schemeAMB} is unconditionally stable in $(0,T]$.
\end{theorem}

\begin{proof}
Setting $\vec{\mathbf{v}}=\vm{u}^{n+1}=\left(\mathbf{u}^{n+1},\phi^{n+1}\right)$
in (\ref{schemeAMB}), we deduce
\begin{equation}\label{AMB2:scheme2}
\begin{aligned}
&\big\langle\big\langle\frac{\vm{u}^{n+1}-\vm{u}^n}{\Delta
t},\vm{u}^{n+1}\big\rangle\big\rangle+a(D_\alpha{\vm{u}}^{n+1},\vm{u}^{n+1}) + a_{st}(D_\alpha{\vm{u}}^{n+1},\vm{u}^{n+1})\\
&\quad=\langle\langle\langle\vec{\mathbf{f}}^{n+\frac{1}{2}},\vm{u}^{n+1}\rangle\rangle\rangle
-a_\Gamma\Big(\frac{3}{2}\vm{u}^{n}-\frac{1}{2}\vm{u}^{n-1},\vm{u}^{n+1}\Big)
+a_{st}\Big(\frac{3}{2}\vm{u}^{n}-\frac{1}{2}\vm{u}^{n-1},\vm{u}^{n+1}\Big).
\end{aligned}
\end{equation}
Combining the two $a_{st}(\cdot, \cdot)$ terms and using the basic equality $2(a-b)a=|a|^2-|b|^2+|a-b|^2$, we have
\begin{equation}\label{weakformAMB}
\begin{aligned}
&\frac{1}{\Delta t}(\|\vm{u}^{n+1}\|_S^2-\|\vm{u}^{n}\|_S^2+\|\vm{u}^{n+1}-\vm{u}^{n}\|_S^2) +2a\left(D_\alpha{\vm{u}}^{n+1},\vm{u}^{n+1}\right)
\\&\quad=2(\vec{\mathbf{f}}^{n+\frac{1}{2}},\vm{u}^{n+1})
 -2a_\Gamma\left(\frac{3}{2}\vm{u}^{n}-\frac{1}{2}\vm{u}^{n-1},\vm{u}^{n+1}\right)  -2\alpha
a_{st}(\delta \vm{u}^{n+1},\vm{u}^{n+1}).
\end{aligned}
\end{equation}
Note that $\alpha>\alpha_1+\alpha_2$ provided $\frac{1}{2}<\alpha<1$.
Therefore, $\beta_1=2\alpha- \beta_3= 2\alpha - (\alpha_1+\alpha_2)>\alpha$ and hence $\beta_1>\beta_2>\beta_3$ when $\half <\alpha<1$. By the Cauchy-Schwarz inequality, we then have
\begin{equation}\label{goodtermAMB}
\begin{aligned}
&2  a\left(D_\alpha{\vm{u}}^{n+1},\vm{u}^{n+1}\right) \ge 2 \alpha a(\vm{u}^{n+1}, \vm{u}^{n+1}) -
\alpha_1(a(\vm{u}^{n+1},\vm{u}^{n+1}) +a(\vm{u}^{n}, \vm{u}^{n})) \\
&\qquad\qquad -\alpha_2(a(\vm{u}^{n+1},\vm{u}^{n+1})+a(\vm{u}^{n-1},\vm{u}^{n-1}))\\
  &\qquad = \beta_1a(\vm{u}^{n+1},\vm{u}^{n+1}) -
\alpha_1a(\vm{u}^{n},\vm{u}^{n})- \alpha_2a(\vm{u}^{n-1},\vm{u}^{n-1}).
\end{aligned}
\end{equation}
Similarly as for \eqref{BDF:interface1}, for the interface coupling term, there exists a constant $C_5$ such that
\begin{equation}\label{badterm}
\begin{aligned}
&-2 a_\Gamma\Big(\frac{3}{2}\vm{u}^{n}-\frac{1}{2}\vm{u}^{n-1},\vm{u}^{n+1}\Big)-2\alpha
a_{st}(-2\vm{u}^{n}+\vm{u}^{n-1},\vm{u}^{n+1})\\
&\qquad\le \frac{C_a(\beta_1-\beta_2)}{4} \|\nabla\vm{u}^{n}\|^{2}+2C_5 \|\vm{u}^{n}\|_S^{2}
+\frac{C_a(\beta_1-\beta_2)}{4}\|\nabla\vm{u}^{n+1}\|^2\\
&\qquad\qquad +\frac{C_a(\beta_1-\beta_2)}{8} \|\nabla\vm{u}^{n-1}\|^{2}
 +C_5 \|\vm{u}^{n-1}\|_S^{2}+\frac{C_a(\beta_1-\beta_2)}{8}\|\nabla\vm{u}^{n+1}\|^2.
\end{aligned}
\end{equation}
For the forcing term, there exists a constant $C_6$ such that
\begin{equation}\label{forcingAMB}
2\Delta
t\big\langle\big\langle\big\langle\vec{\mathbf{f}}^{n+\frac{1}{2}},\vm{u}^{n+1}\big\rangle\big\rangle\big\rangle
\le C_6\Delta
t\|\vec{\mathbf{f}}^{n+\frac{1}{2}}\|^2+\frac{C_a (\beta_1-\beta_2)}{8}\Delta
t \|\nabla \vm{u}^{n+1}\|^2.
\end{equation}
Substituting \eqref{goodtermAMB}--\eqref{forcingAMB} into \eqref{weakformAMB} yields
\begin{equation}\label{shorttimeAMB1}
\begin{aligned}
&\|\vm{u}^{n+1}\|_S^2
+\frac{C_a(\beta_1-\beta_2)}{2}\Delta
t\left\|\nabla\vm{u}^{n+1}\right\|^2+\Delta
t\beta_2\|\vm{u}^{n+1}\|^2_a\\ \
&\quad + \|\vm{u}^{n+1}-\vm{u}^{n}\|_S^2 +2\alpha\Delta t a_{st}\left(\vm{u}^{n+1},\vm{u}^{n+1}\right)\\
 & \le C_6\|\vec{\mathbf{f}}^{n+\frac{1}{2}}\|^2\Delta t+(1+2C_5\Delta t)\|\vm{u}^{n}\|_S^2 +
 C_5\Delta t\|\vm{u}^{n-1}\|_S^2+\Delta
t\alpha_1\|\vm{u}^{n}\|^2_a
\\&\quad+\Delta
t\alpha_2 \|\vm{u}^{n-1}\|^2_a+\frac{C_a(\beta_1-\beta_2)}{4}\Delta
t \|\nabla\vm{u}^{n}\|^2+\frac{C_a(\beta_1-\beta_2)}{8}\Delta
t\|\nabla\vm{u}^{n-1}\|^2.
\end{aligned}
\end{equation}
Define the energy
$$
\begin{aligned}
&E_n= \|\vm{u}^{n}\|_S^2+\frac{C_5\Delta t}{1+3C_5\Delta t}\|\vm{u}^{n-1}\|_S^2+\frac{\beta_3\Delta
t}{1+3C_5\Delta t}\|\vm{u}^{n}\|^2_a+\frac{\Delta
t\alpha_2}{1+3C_5\Delta t}\|\vm{u}^{n-1}\|^2_a\\
&\qquad\qquad +\frac{3C_a(\beta_1-\beta_2)}{8(1+3C_5\Delta t)}\Delta
t\|\nabla\vm{u}^{n}\|^2+\frac{C_a(\beta_1-\beta_2)}{8(1+3C_5\Delta t)}\Delta
t\|\nabla\vm{u}^{n-1}\|^2.
\end{aligned}
$$
Then, discarding the last two positive terms on the left-hand side of \eqref{shorttimeAMB1} and adding $C_5\Delta t \|\vm{u}^{n}\|_S^2+\alpha_2\Delta t\|\vm{u}^{n}\|^2_a+\frac{C_a(\beta_1-\beta_2)}{8}\Delta t\|\nabla\vm{u}^{n}\|^2$ to both sides, we obtain
$$
\begin{aligned}
&E_{n+1}+\frac{3C_5\Delta t^2}{1+3C_5\Delta t}\|\vm{u}^{n}\|_S^2 +\frac{(\beta_2-\beta_3)\Delta
t+3C_5\beta_3\Delta t^2}{1+3C_5\Delta t}\|\vm{u}^{n+1}\|^2_a\\
&\qquad +\frac{3C_5\alpha_2\Delta
t^2}{1+3C_5\Delta t}\|\vm{u}^{n}\|^2_a
+\frac{3C_5C_a(\beta_1-\beta_2)\Delta t^2}{8(1+3C_5\Delta t)}\|\nabla\vm{u}^{n}\|^2\\
&\qquad +\frac{C_a(\beta_1-\beta_2)\Delta
t+12C_5C_a(\beta_1-\beta_2)\Delta t^2}{8(1+3C_5\Delta t)}\|\nabla\vm{u}^{n+1}\|^2\\
& \le   C_6\|\vec{\mathbf{f}}^{n+\frac{1}{2}}\|^2\Delta t +(1+3C_5\Delta t)E_n.
\end{aligned}
$$
Discarding all terms on the left-hand side, all of which are positive, except for $E_{n+1}$, we are left with
$$
E_{n+1}\le  C_6 \Delta t\max_{n}\|\vec{\mathbf{f}}^{n+\frac{1}{2}}\|^2  +(1+3C_5\Delta t)E_n.
$$
Then, by recursion,
$$
E_n \le e^{3C_5 T} E_1+ \frac{ C_6  }{3C_5}e^{3C_5 T}\max_{i}\|\vec{\mathbf{f}}^{i+\frac{1}{2}}\|^2
$$
so that the proof of the theorem is complete.
\end{proof}

\subsection{Long-time stability of the the BDF2 and AMB2 schemes}\label{sec:32}

\subsubsection{Uniform in time estimates for the BDF2 scheme}\label{sec:321}

\textcolor{white}{xxxx}

\begin{theorem}\label{Longstab:BDF2}
Assume that $\vm{f}\in L^\infty(L^2(\Omega))$ and that the time-step restriction \eqref{timestepBDF} is satisfied. Then, the solution to the {\em BDF2} scheme \eqref{schemeBDF} is uniformly bounded for all time.
Specifically, there exist $0<\lambda_1<1$, $\lambda_2<\infty$, and $E_0\ge 0$ such that
$$
\|\vm{u}^{n}\|^2\le C_u \lambda_1^n  E_0+\lambda_2.
$$
\end{theorem}

\begin{proof}
Recall that $a_{\Gamma}(\vm{u}^{n+1},\vm{u}^{n+1})=0$. Therefore, by Lemma \ref{Lemma1},
\begin{equation}\label{BDF:agamma}
\begin{aligned}
a_\Gamma(-2\vm{u}^{n}&+\vm{u}^{n-1},\vm{u}^{n+1}) -a_{st}(\delta \vm{u}^{n+1},\vm{u}^{n+1})\\& =
\widetilde{a}_\Gamma(\delta \vm{u}^{n+1},\vm{u}^{n+1})
\le  C_{ct} \left\|\delta \vm{u}^{n+1}\right\|_\Gamma\left\|\vm{u}^{n+1}\right\|_\Gamma.
\end{aligned}
\end{equation}
The trace and Poincar\'{e} inequalities imply
\begin{equation}\label{BDF:agamma2}
\begin{aligned}
&\left\|\delta \vm{u}^{n+1}\right\|_\Gamma\left\|\vm{u}^{n+1}\right\|_\Gamma  \le  C_{tr}^2 \|\delta \vm{u}^{n+1}\|^\half\|\nabla \delta \vm{u}^{n+1}\|^\half
\|\nabla\vm{u}^{n+1}\| \\
&\quad\le C_S^\half C_{tr}^2 \|\delta \vm{u}^{n+1}\|^\half_S(\|\nabla\vm{u}^{n+1}\|^\half
+\sqrt{2}\|\nabla \vm{u}^{n}\|^\half+\|\nabla \vm{u}^{n-1}\|^\half)
\|\nabla\vm{u}^{n+1}\| .
\end{aligned}
\end{equation}
The three terms on the right-hand side can be bounded using Young's inequalities:
$$
\begin{aligned}
 \|\delta \vm{u}^{n+1}\|^\half_S\|\nabla\vm{u}^{n+1}\|^\frac{3}{2}
&\le \frac{\varepsilon}{8}\|\nabla \vm{u}^{n+1}\|^2 + \frac{{54}}{\varepsilon^3}\|\delta \vm{u}^{n+1}\|^2_S\\
\sqrt{2}\|\delta \vm{u}^{n+1}\|^\half_S\|\nabla\vm{u}^{n}\|^\frac{1}{2}\|\nabla\vm{u}^{n+1}\|&\le \frac{\varepsilon}{8}\|\nabla \vm{u}^{n+1}\|^2
+\frac{\varepsilon}{16}\|\nabla \vm{u}^{n}\|^2+\frac{64}{\varepsilon^3}\|\delta \vm{u}^{n+1}\|^2_S\\
\|\delta \vm{u}^{n+1}\|^\half_S\|\nabla\vm{u}^{n-1}\|^\frac{1}{2}\|\nabla\vm{u}^{n+1}\|&\le \frac{\varepsilon}{8}\|\nabla \vm{u}^{n+1}\|^2
+\frac{\varepsilon}{16}\|\nabla \vm{u}^{n-1}\|^2+\frac{16}{\varepsilon^3}\|\delta \vm{u}^{n+1}\|^2_S.
\end{aligned}
$$
Then, setting $\varepsilon=\varepsilon_0=\frac{C_a}{C_s^\half C_{ct}C_{tr}^2}$, we deduce from these three inequalities, \eqref{BDF:agamma}, and \eqref{BDF:agamma2} that
\begin{equation}\label{actBDF}
\begin{aligned}
 &\widetilde{a}_\Gamma(\delta \vm{u}^{n+1},\vm{u}^{n+1})\le \frac{3C_a}{8}\left\|\nabla\vm{u}^{n+1}\right\|^2 \\
&\qquad +\frac{C_a}{16}\left\|\nabla\vm{u}^{n}\right\|^2+\frac{C_a}{16}\left\|\nabla\vm{u}^{n-1}\right\|^2
+\frac{{134}C_s^2C_{ct}^4C_{tr}^8}{C_a^3}\left\|\delta \vm{u}^{n+1}\right\|^2_S.
\end{aligned}
\end{equation}
The forcing term can be bounded as
\begin{equation}\label{forcingBDFlong}
 \big\langle\big\langle\big\langle \vm{f}^{n+1},{\mathbf{u}}^{n+1}\big\rangle\big\rangle\big\rangle
 \le \frac{2{C^2_P}}{C_a}\|\vm{f}^{n+1}\|^2+\frac{C_a }{8{C^2_P}}\left\|{\vm{u}^{n+1}}\right\|^2.
\end{equation}
Combining \eqref{ineq1} and \eqref{ineq2} with \eqref{actBDF} and \eqref{forcingBDFlong}, we obtain
$$
\begin{aligned}
&|\vec{\mathbf{w}}_n|^2_G+C_a\Delta
t\left\|\nabla\vm{u}^{n+1}\right\|^2+
\left[\frac{1}{2}- \frac{{268}C_s^2C_{ct}^4C_{tr}^8}{C_a^3}\Delta
t\right]\left\|\delta \vm{u}^{n+1}\right\|_S^2\\
&\qquad\le\frac{4{C^2_P}\Delta t}{C_a}\|\vm{f}^{n+1}\|^2+|\vec{\mathbf{w}}_{n-1}|^2_G+\frac{C_a\Delta
t}{8}\left\|\nabla\vm{u}^{n}\right\|^2+\frac{C_a\Delta
t}{8}\left\|\nabla\vm{u}^{n-1}\right\|^2.   \label{BDF:rec1}
\end{aligned}
$$
If the time-step restriction
\begin{equation}\label{timestepBDF}
\Delta
t\le\frac{C_a^3}{{536}C_s^2C_{ct}^4C_{tr}^8}.
\end{equation}
is satisfied, this leads to
$$
\begin{aligned}
 &|\vec{\mathbf{w}}_n|^2_G+C_a\Delta
t\left\|\nabla\vm{u}^{n+1}\right\|^2 \\&\quad\le
\frac{4{C^2_P}\Delta t}{C_a}\|\vm{f}^{n+1}\|^2 +|\vec{\mathbf{w}}_{n-1}|^2_G+\frac{C_a\Delta
t}{8}\left\|\nabla\vm{u}^{n}\right\|^2
+\frac{C_a\Delta
t}{8}\left\|\nabla\vm{u}^{n-1}\right\|^2.
\end{aligned}
$$
Adding $\frac{3C_a\Delta t}{8}\left\|\nabla\vm{u}^{n}\right\|^2$ to both sides of the above inequality,  we obtain
$$
\begin{aligned}
&|\vec{\mathbf{w}}_n|^2_G+C_a\Delta
t\left\|\nabla\vm{u}^{n+1}\right\|^2+\frac{3C_a\Delta
t}{8}\left\|\nabla\vm{u}^{n}\right\|^2\\ &\quad\le
\frac{4{C^2_P}\Delta t}{C_a}\|\vm{f}^{n+1}\|^2 +|\vec{\mathbf{w}}_{n-1}|^2_G+\frac{C_a\Delta
t}{2}\left\|\nabla\vm{u}^{n}\right\|^2+\frac{C_a\Delta
t}{8}\left\|\nabla\vm{u}^{n-1}\right\|^2
\end{aligned}
$$
which is equivalent to
\begin{equation}\label{enineq}
E_{n}+\frac{C_a}{2}\Delta
t\left\|\nabla\vm{u}^{n+1}\right\|^2+\frac{C_a}{4}\Delta
t\left\|\nabla\vm{u}^{n}\right\|^2\le E_{n-1}+\frac{4{C^2_P}\Delta t}{C_a}\|\vm{f}^{n+1}\|^2,
\end{equation}
where
$E_{n}=|\vec{\mathbf{w}}_{n}|^2_G+\frac{C_a\Delta
t}{2}\left\|\nabla\vm{u}^{n+1}\right\|^2+\frac{C_a\Delta
t}{8}\left\|\nabla\vm{u}^{n}\right\|^2$.

Utilizing the Poincar\'e inequality and the equivalence of the $G$-norm and the $L^2$-norm, we have
$$
\frac{C_a}{2}\left\|\nabla\vm{u}^{n+1}\right\|^2+\frac{C_a}{4}\left\|\nabla\vm{u}^{n}\right\|^2
\ge  \frac{C_a}{4}\left\|\nabla\vm{u}^{n+1}\right\|^2+\frac{C_a}{8}\left\|\nabla\vm{u}^{n}\right\|^2
+\frac{C_l^2C_a}{8C_P^2}|\vm{w}_n|_G^2.
$$
Therefore, setting $C_7=\min\{\frac{C_l^2C_a}{8C_P^2},\frac{1}{2\Delta t}\}$, we have from \eqref{enineq} that
$$
(1+C_7\Delta t)E_{n}\le  E_{n-1}+\frac{4{C^2_P}\Delta t}{C_a}\|\vm{f}^{n+1}\|^2.
$$
A simple induction argument leads to
$$
E_{n}\le \Big(\frac{1}{1+C_7\Delta t}\Big)^n E_0+\frac{4{C^2_P}(1+C_7\Delta t)}{C_aC_7}\max_{i}\|\vm{f}^{i}\|^2.
$$
Recall that $\|\vm{u}^n\|\le C_u  E_{n}$. Hence, the theorem is proved with $\lambda_1=\frac{1}{1+C_7\Delta t}$ and $\lambda_2=C_u\frac{4{C^2_P}(1+C_7\Delta t)}{C_aC_7}\max_{i}\|\vm{f}^{i}\|^2$.
\end{proof}

The following corollary is used in the analysis of the fully-discrete BDF2 scheme; see \S\ref{sec:51}.

\begin{corollary}\label{Longstab:BDF2-1}
In addition to the assumptions of Theorem~{\em\ref{Longstab:BDF2}}, assume that the second time-step restriction \eqref{timestepBDF2} is satisfied. Then,
\begin{equation}
\|\vm{u}^{n}\|^2\le C \lambda_1^{n-2}\big( \|\vm{u}^0\|^2+\|\vm{u}^1\|^2+\Delta t^2 \|\nabla\vm{u}^0\|^2+\Delta t^2 \|\nabla\vm{u}^1\|^2\big)+C\lambda_2.
\end{equation}
\end{corollary}

\begin{proof}
For the interface term \eqref{BDF:agamma}, we can derive another estimate.
From \eqref{Inq:trace} and noting that $\|\nabla\delta \vm{u}^{n+1}\|^\half \le \sqrt{2}(\|\nabla\vm{u}^{n+1}\|^\half+ \|\nabla\vm{u}^{n}\|^\half+\|\nabla\vm{u}^{n-1}\|^\half)$ and $\|\vm{u}^{n+1}\|^\half \le \sqrt{2}(\|\delta\vm{u}^{n+1}\|^\half+ \|\vm{u}^{n}\|^\half+\|\vm{u}^{n-1}\|^\half)$, we have
$$
\begin{aligned}
&\widetilde{a}_{\Gamma}(\delta \vm{u}^{n+1},\vm{u}^{n+1})
\\&\quad
\le\widetilde{C}\|\delta \vm{u}^{n+1}\|^\half_S
\sum_{j=n-1}^{n+1}\|\nabla \vm{u}^{j}\|^\half
(\|\delta\vm{u}^{n+1}\|^\half_S+ \|\vm{u}^{n}\|^\half_S+\|\vm{u}^{n-1}\|^\half_S)\|\nabla \vm{u}^{n+1}\|^\half \\
&\quad=
\widetilde{C}\|\delta \vm{u}^{n+1}\|_S
\sum_{j=n-1}^{n+1}\|\nabla \vm{u}^{j}\|^\half \|\nabla \vm{u}^{n+1}\|^\half \\
&\qquad + \widetilde{C} \|\delta \vm{u}^{n+1}\|^\half_S
\sum_{j=n-1}^{n+1}\|\nabla \vm{u}^{j}\|^\half   ( \|\vm{u}^{n}\|^\half_S+\|\vm{u}^{n-1}\|^\half_S)\|\nabla \vm{u}^{n+1}\|^\half
:=S_1+S_2,
\end{aligned}
$$
where $\widetilde{C}=2C_sC_{ct}C_{tr}^2$. The terms in the right-hand side can be bounded by Young's inequalities:
$$
\begin{aligned}
S_1&\le \sum_{j=n-1}^{n+1} \left( \frac{1}{24\Delta t}\|\delta \vm{u}^{n+1}\|^2_S+
 {3\widetilde{C}^2\Delta t} (\|\nabla \vm{u}^{n+1}\|^2 +  \|\nabla \vm{u}^{j}\|^2)\right)\\
&=  \frac{1}{8\Delta t}\|\delta \vm{u}^{n+1}\|^2_S+
 {3\widetilde{C}^2\Delta t}  (\|\nabla \vm{u}^{n-1}\|^2+\|\nabla \vm{u}^{n}\|^2+2\|\nabla \vm{u}^{n+1}\|^2)
\end{aligned}
$$
and
$$
\begin{aligned}
  S_2 &\le  \sum_{j=n-1}^{n+1}\sum_{k=n-1}^{n}\Big(\frac{1}{48\Delta t}\|\delta \vm{u}^{n+1}\|^2_S+
   \frac{C_a}{24}\|\nabla \vm{u}^{n+1}\|^2+ \frac{9\widetilde{C}^2}{C_a} \|\vm{u}^{k}\|^2_S
  +\frac{\widetilde{C}^2\Delta t}{2}  \|\nabla \vm{u}^{j}\|^2\Big)\\
 &=  \frac{1}{8\Delta t}\|\delta \vm{u}^{n+1}\|^2_S +  \frac{C_a}{4}\|\nabla \vm{u}^{n+1}\|^2 + \sum_{k=n-1}^n\frac{27\widetilde{C}^2}{C_a} \|\vm{u}^{k}\|^2_S
  +\sum_{j=n-1}^{n+1} {\widetilde{C}^2\Delta t}   \|\nabla \vm{u}^{j}\|^2.
\end{aligned}
$$
Now, if we require
\begin{equation}
\Delta t\le \frac{3C_a}{56\widetilde{C}^2}=\frac{3C_a}{112C_s^2C_{ct}^2C_{tr}^4}, \label{timestepBDF2}
\end{equation}
the interface term can then be bounded by
$$
\begin{aligned}
  & \widetilde{a}_{\Gamma}(\delta \vm{u}^{n+1},\vm{u}^{n+1}) \le
   \frac{1}{4\Delta t}\|\delta \vm{u}^{n+1}\|^2_S
   \\&\quad+  \frac{5C_a}{8}\|\nabla \vm{u}^{n+1}\|^2
   +\frac{27\widetilde{C}^2}{C_a} (\|\vm{u}^{n}\|^2_S+\|\vm{u}^{n-1}\|^2_S)  + 4\widetilde{C}^2\Delta t  (\|\nabla \vm{u}^{n}\|^2+\|\nabla \vm{u}^{n-1}\|^2)
\end{aligned}
$$
which leads to another recursion formula:
$$
\begin{aligned}
&|\vec{\mathbf{w}}_n|^2_G+\frac{C_a\Delta t}{2}
\left\|\nabla\vm{u}^{n+1}\right\|^2 \le\frac{4{C^2_P}\Delta t}{C_a}\|\vm{f}^{n+1}\|^2\\&\qquad
+|\vec{\mathbf{w}}_{n-1}|^2_G +\frac{54\widetilde{C}^2\Delta t}{C_a} (\|\vm{u}^{n}\|^2_S+\|\vm{u}^{n-1}\|^2_S)
   + 8\widetilde{C}^2\Delta t^2  (\|\nabla \vm{u}^{n}\|^2+\|\nabla \vm{u}^{n-1}\|^2).
\end{aligned}
$$
Using this relationship, it is easy to verify that
$$
  E_n\le C\Delta t(\|\vm{f}^{n+1}\|^2 +\|\vm{f}^{n}\|^2)+ C|V_{n-2}|_G^2 +C\Delta t^2 (\|\nabla \vm{u}^{n-1}\|^2+\|\nabla \vm{u}^{n-2}\|^2).
$$
Specifically, for $n=2$, we have
\begin{equation}\label{Longstab:E2}
  E_2\le C\Delta t(\|\vm{f}^{2}\|^2 +\|\vm{f}^{1}\|^2)+ C|V_{0}|_G^2 +C\Delta t^2 (\|\nabla \vm{u}^{0}\|^2+\|\nabla \vm{u}^{1}\|^2).
\end{equation}
Combing with Theorem \ref{Longstab:BDF2} completes the proof.
\end{proof}

\subsubsection{Uniform in time estimates for the AMB2 scheme}\label{sec322}

We start with the following estimate.

\begin{lemma}\label{AMB2:Egamma}
Let
$$
{\mathcal E}_\Gamma=-2 a_\Gamma\left(\frac{3}{2}\vm{u}^{n}-\frac{1}{2}\vm{u}^{n-1},\vm{u}^{n+1}\right)-2\alpha
a_{st}\left(\delta \vm{u}^{n+1},\vm{u}^{n+1}\right).
$$
Then, with $\beta_1$ and $\beta_2$ defined in \eqref{alpha-beta}, we have the bound
\begin{equation}\label{newbound}
\begin{aligned}
  &|{\mathcal E}_{\Gamma}|\le \frac{4C_a(\beta_1-\beta_2)}{9} \|\nabla\vm{u}^{n+1}\|^2+\frac{2C_a(\beta_1-\beta_2)}{9}
  \|\nabla\vm{u}^{n}\|^2
 \\&\quad+\frac{C_a(\beta_1-\beta_2)}{9}
  \|\nabla\vm{u}^{n-1}\|^2
+(C_8+C_9)
\|\vm{u}^{n+1}-\vm{u}^{n}\|_S^2
+2C_9\|\vm{u}^{n}-\vm{u}^{n-1}\|_S^2.
\end{aligned}
\end{equation}
\end{lemma}

\begin{proof}
Recall that $a_\Gamma(\cdot,\cdot)$ is skew-symmetric. Therefore,
\begin{equation}\label{AMB2:E}
\begin{aligned}
|{\mathcal E}_{\Gamma}|&\le |2  a_\Gamma\left(\vm{u}^{n+1}-\vm{u}^{n},\vm{u}^{n+1}\right)-
a_\Gamma\left(\vm{u}^{n}-\vm{u}^{n-1},\vm{u}^{n+1}\right)|\\
&\quad+ |-2\alpha
a_{st}\left(\vm{u}^{n+1}-\vm{u}^{n},\vm{u}^{n+1}\right)+2 \alpha
a_{st}\left(\vm{u}^{n}-\vm{u}^{n-1},\vm{u}^{n+1}\right)|\\
&\le 2C_{ct}
\|\vm{u}^{n+1}-\vm{u}^{n}\|_\Gamma\|\vm{u}^{n+1}\|_\Gamma+2C_{ct}
\|\vm{u}^{n}-\vm{u}^{n-1}\|_\Gamma\|\vm{u}^{n+1}\|_\Gamma\\
&\le 2C_{ct}C^2_{tr} \|\vm{u}^{n+1}-\vm{u}^{n}\|^{1/2}\|\nabla\left(\vm{u}^{n+1}-\vm{u}^{n}\right)\|^{1/2}\|\nabla\vm{u}^{n+1}\|\\
&\quad+2C_{ct}C^2_{tr}\|\vm{u}^{n}-\vm{u}^{n-1}\|^{1/2}\|\nabla\left(\vm{u}^{n}-\vm{u}^{n-1}\right)\|^{1/2}\|\nabla\vm{u}^{n+1}\|\\
&\le 2C_{ct}C^2_{tr} \|\vm{u}^{n+1}-\vm{u}^{n}\|^{1/2}\|\nabla\vm{u}^{n+1}\|\left(\|\nabla\vm{u}^{n+1}\|^{1/2}+\|\nabla\vm{u}^{n}\|^{1/2}\right)\\
&\quad+2C_{ct}C^2_{tr} \|\vm{u}^{n}-\vm{u}^{n-1}\|^{1/2}\|\nabla\vm{u}^{n+1}\|\left(\|\nabla\vm{u}^{n}\|^{1/2}+\|\nabla\vm{u}^{n-1}\|^{1/2}\right).
\end{aligned}
\end{equation}
Then, by Young's inequality and \eqref{eq:normS},
$$
\begin{aligned}
|{\mathcal E}_{\Gamma}|&\le \frac{C_a(\beta_1-\beta_2)}{9} \|\nabla\vm{u}^{n+1}\|^2+ C_8 \|\vm{u}^{n+1}-\vm{u}^{n}\|_S^2  \\
&\quad+\frac{C_a(\beta_1-\beta_2)}{9} \|\nabla\vm{u}^{n+1}\|^2+\frac{C_a(\beta_1-\beta_2)}{9} \|\nabla\vm{u}^{n}\|^2+ C_9 \|\vm{u}^{n+1}-\vm{u}^{n}\|_S^2\\
&\quad+\frac{C_a(\beta_1-\beta_2)}{9} \|\nabla\vm{u}^{n+1}\|^2+\frac{C_a(\beta_1-\beta_2)}{9}
\|\nabla\vm{u}^{n}\|^2+C_9 \|\vm{u}^{n}-\vm{u}^{n-1}\|_S^2\\
&\quad+\frac{C_a(\beta_1-\beta_2)}{9} \|\nabla\vm{u}^{n+1}\|^2+\frac{C_a(\beta_1-\beta_2)}{9} \|\nabla\vm{u}^{n-1}\|^2
+C_9 \|\vm{u}^{n}-\vm{u}^{n-1}\|_S^2\\
 &=\frac{4C_a(\beta_1-\beta_2)}{9} \|\nabla\vm{u}^{n+1}\|^2+\frac{2C_a(\beta_1-\beta_2)}{9}
\|\nabla\vm{u}^{n}\|^2\\ &\quad+\frac{ C_a(\beta_1-\beta_2)}{9}\|\nabla\vm{u}^{n-1}\|^2
+(C_8+C_9)
\|\vm{u}^{n+1}-\vm{u}^{n}\|_S^2
 +2C_9\|\vm{u}^{n}-\vm{u}^{n-1}\|_S^2.
\end{aligned}
$$
\end{proof}

\begin{theorem} \label{Longstab:AMB2}
Assume that $1/2<\alpha<1$, $\vm{f}\in L^\infty(L^2(\Omega_f))$, and that the time-step restriction \eqref{timestepAMB} is satisfied.
Then, the  solution to the {\em AMB2} scheme \eqref{schemeAMB} is uniformly bounded for all time. Specifically, there exist  $0<\lambda_3<1$, $\lambda_4<\infty$, and $E_1\ge 0$ such that
$$
\|\vm{u}^{n+1}\|^2\le C_S \lambda_3^n E_1+\lambda_4.
$$
\end{theorem}
\begin{proof} The interface term has been estimated in Lemma \ref{AMB2:Egamma}. The forcing term can be bounded as
\begin{equation}\label{forcingAMB2}
2 \big\langle\big\langle\big\langle\vm{f}^{n+\frac{1}{2}},\vm{u}^{n+1}\big\rangle\big\rangle\big\rangle
\le C_{10}\|\vm{f}^{n+\frac{1}{2}}\|^2+\frac{C_a(\beta_1-\beta_2)}{9} \|\nabla\vm{u}^{n+1}\|^2.
\end{equation}
Combining \eqref{goodtermAMB}, \eqref{newbound}, and \eqref{forcingAMB2}, \eqref{weakformAMB} becomes
\begin{equation}\label{longtimeAMB1}
\begin{aligned}
&\|\vm{u}^{n+1}\|_S^2+\Delta
t\beta_2\|\vm{u}^{n+1}\|^2_a
+\frac{4C_a(\beta_1-\beta_2)}{9}\Delta
t\|\nabla\vm{u}^{n+1}\|^2\\
&\qquad\qquad  +(1-(C_{8}+C_{9})\Delta
t)\|\vm{u}^{n+1}-\vm{u}^{n}\|_S^2\\
 &\qquad\le \|\vm{u}^{n}\|_S^2 +  C_{10}\|\vm{f}^{n+\frac{1}{2}}\|^2\Delta t+\Delta
t\alpha_1\|\vm{u}^{n}\|^2_a+\Delta
t\alpha_2\|\vm{u}^{n-1}\|^2_a\\
&\qquad\qquad+\frac{2C_a(\beta_1-\beta_2)}{9}\Delta
t\|\nabla\vm{u}^{n}\|^2+\frac{C_a(\beta_1-\beta_2)}{9}\Delta
t\|\nabla\vm{u}^{n-1}\|^2\\&
\qquad\qquad +2C_{9}\Delta
t\|\vm{u}^{n}-\vm{u}^{n-1}\|_S^2.
\end{aligned}
\end{equation}
Now add $C_{11}\Delta t\|\vm{u}^{n}\|^2_a+C_{12}\frac{C_a(\beta_1-\beta_2)}{9}\Delta t\|\nabla\vm{u}^{n}\|^2$ to both sides, require that
$$
\beta_2-\alpha_1>C_{11}>\alpha_2, \quad 2>C_{12}>1,
$$
require the time-step restriction
\begin{equation}\label{timestepAMB}
\Delta t<\frac{1}{C_{8}+3C_{9}},
\end{equation}
and set
$$
\begin{aligned}
E_n=&\|\vm{u}^{n}\|_S^2+\Delta
t\left(\alpha_1+C_{11}\right)\|\vm{u}^{n}\|^2_a+(2+C_{12})\frac{C_a(\beta_1-\beta_2)}{9}\Delta
t\|\nabla\vm{u}^{n}\|^2\\
&\qquad+\Delta
t\alpha_2\|\vm{u}^{n-1}\|^2_a+\frac{C_a(\beta_1-\beta_2)}{9}\Delta
t\|\nabla\vm{u}^{n-1}\|^2+2C_{9}\Delta
t \|\vm{u}^{n}-\vm{u}^{n-1}\|_S^2.
\end{aligned}
$$
Then, \eqref{longtimeAMB1} becomes
\begin{equation}\label{longtimeAMB2}
\begin{aligned}
E_{n+1}&+\Delta t\left(\beta_2-\alpha_1-C_{11}\right)\|\vm{u}^{n+1}\|^2_a
+(2-C_{12})\frac{C_a(\beta_1-\beta_2)}{9}\Delta
t\|\nabla\vm{u}^{n+1}\|^2\\
&\qquad+\Delta t\left(C_{11}-\alpha_2\right)\|\vm{u}^{n}\|^2_a  + (C_{12}-1)\frac{C_a(\beta_1-\beta_2)}{9}\Delta
t\|\nabla\vm{u}^{n}\|^2 \\
&\qquad+ (1-(C_{8}+3C_{9})\Delta t)\|\vm{u}^{n+1}-\vm{u}^{n}\|_S^2 \le  C_{10}\|\vm{f}^{n+\frac{1}{2}}\|^2 \Delta t+E_n.
\end{aligned}
\end{equation}
Because there exists a constant $C_{13}>0$ such that
$$
\begin{aligned}
C_{13}\|\vm{u}^{n+1}\|_S  &\le  (2-C_{12})\frac{C_a(\beta_1-\beta_2)}{18}\|\nabla \vm{u}^{n+1}\|
\\
C_{13}\Delta
t^2 \left(\alpha_1+C_{11}\right)&\le \Delta t\left(\beta_2-\alpha_1-C_{11}\right)\\
C_{13}\Delta
t^2\alpha_2&\le\Delta t\left(C_{11}-\alpha_2\right)\\
C_{13}(2+C_{12})\frac{C_a(\beta_1-\beta_2)}{9}\Delta
t^2&\le(2-C_{12})\frac{C_a(\beta_1-\beta_2)}{18}\Delta
t,\\
C_{13}\frac{C_a(\beta_1-\beta_2)}{9}\Delta
t^2&\le(C_{12}-1)\frac{C_a(\beta_1-\beta_2)}{9}\Delta
t,\\
2C_{13}C_{9}\Delta
t^2&\le (1-(C_{8}+3C_{9})\Delta t)
\end{aligned}
$$
we have from \eqref{longtimeAMB2} that
$$
(1+C_{13}\Delta t)E_{n+1}\le E_n + C_{10}\|\vm{f}^{n+\frac{1}{2}}\|^2\Delta t.
$$
Thus, we have
$$
\|\vm{u}^{n+1}\|_S^2\le  E_{n+1}  \le \Big(\frac{1}{1+C_{13}\Delta t}\Big)^n E_1
 +\frac{ C_{10}(1+C_{13}\Delta t)}{C_{13}}\max_i\|\vm{f}^{i+\frac{1}{2}}\|^2.
$$
Setting $\lambda_3=\frac{1}{1+C_{13}\Delta t}$ and $ \lambda_4=\frac{C_{10}(1+C_{13}\Delta t)}{C_{13}}\max_i\|\vm{f}^{i+\frac{1}{2}}\|^2$, by \eqref{eq:normS} the proof is complete.
\end{proof}

\begin{rmk}
Similarly to Corollary~{\rm\ref{Longstab:BDF2-1}}, in the error analysis, $E_1$ can be taken as $C(|\vm{w}_0|_G+\Delta t^2(\|\nabla \vm{u}_0\|^2+\|\nabla \vm{u}_1\|^1))$ in Theorem~{\em\ref{Longstab:AMB2}}.
\end{rmk}

\section{$H^1(\Omega)$ stability of the schemes}\label{sec:4}

The purpose of this section is to prove uniform in time bounds for the solutions to the schemes \eqref{schemeBDF} and \eqref{schemeAMB} with respect to the $H^1(\Omega)$ norm. This additional estimate is needed for the estimation of finite element element errors for for the fluid velocity and hydraulic head with respect to the $H^1(\Omega)$ and for the pressure with respect to the $L^2(\Omega_f)$ norm; see \S\ref{sec:51}.

\subsection{Uniform in time $H^1(\Omega)$ bound of the BDF2 scheme}\label{sec:41}

In this subsection, we assume that the time-step restriction \eqref{timestepBDF} holds. We introduce the notation $\bar\partial \vm{u}^{n+1}=\frac{1}{\Delta t}(\vm{u}^{n+1}-\vm{u}^n)$.

\begin{lemma} \label{lem:BDF:barpartialu}
The first-order discrete time derivative of the {\em BDF2} scheme \eqref{schemeBDF} is uniformly bounded in time. Specifically, we have
\begin{equation}\label{BDF:barpartialu}
   \|\bar{\partial}\vm{u}^{n+1}\|^2\le C \lambda_1^{n} +C\max_i\|\bar{\partial}\vm{f}^{i}\|^2,
\end{equation}
where the positive parameter $\lambda_1<1$ is defined in Theorem~{\em\ref{Longstab:BDF2}}.
\end{lemma}

\begin{proof} For the BDF scheme \eqref{schemeBDF}, we take the difference of the $n$ and $n+1$ level equations to obtain
$$
\begin{aligned}
 &\frac{1}{\Delta t}\big\langle\big\langle D \bar{\partial}\vm{u}^{n+1},\vm{v}\big\rangle\big\rangle +a(\bar{\partial}\vm{u}^{n+1}, \vm{v})
 +b(\mathbf{v},\bar{\partial}p^{n+1})
 +a_{st}(\delta \bar{\partial}\vm{u}^{n+1}, \vm{v})\nonumber\\
 &\qquad\qquad=\big\langle\big\langle\big\langle\bar{\partial}\vm{f}^{n+1}, \vm{v}\big\rangle\big\rangle\big\rangle+a_{\Gamma}(-2\bar{\partial}\vm{u}^n+\bar{\partial}\vm{u}^{n-1},\vm{v}).
\end{aligned}
$$
Now setting $\vm{v}=\bar{\partial}\vm{u}^{n+1}$ and using the skew-symmetry of $a_{\Gamma}$, we have
$$
\begin{aligned}
 &\frac{1}{\Delta t}\big\langle\big\langle  D \bar{\partial}\vm{u}^{n+1},\bar{\partial}\vm{u}^{n+1}\big\rangle\big\rangle
  +a(\bar{\partial}\vm{u}^{n+1}, \bar{\partial}\vm{u}^{n+1})
 +a_{st}(\delta \bar{\partial}\vm{u}^{n+1}, \bar{\partial}\vm{u}^{n+1})\nonumber\\
 &\qquad\qquad=\big\langle\big\langle\big\langle\bar{\partial}\vm{f}^{n+1}, \bar{\partial}\vm{u}^{n+1}\big\rangle\big\rangle\big\rangle+a_{\Gamma}(\delta\bar{\partial}\vm{u}^{n+1},\bar{\partial}\vm{u}^{n+1}).
\end{aligned}
$$
The rest proof is a verbatim copy of the proof of Theorem \ref{Longstab:BDF2} with $\vm{f}$ replaced by $\bar{\partial}\vm{f}$.
\end{proof}

A direct consequence of the Lemma \ref {lem:BDF:barpartialu} is the following result, once we realize that $\frac{1}{\Delta t}D\vm{u}^{n+1}=\frac{3}{2}\bar{\partial}\vm{u}^{n+1}-\half \bar{\partial} \vm{u}^n$ and $\frac{1}{\Delta t}\delta \vm{u}^{n+1} = \bar{\partial}\vm{u}^{n+1}- \bar{\partial} \vm{u}^n$.

\begin{corollary}
Let $\vm{u}^n$ be the solution to  the {\em BDF2} scheme \eqref{schemeBDF}. Then,
\begin{equation}\label{BDF:Dunorm}
   \|\frac{1}{\Delta t}D \vm{u}^{n+1} \|^2 +\|\frac{1}{\Delta t}\delta \vm{u}^{n+1}\|^2\le
  C \lambda_1^{n} +C \max_i\|\bar{\partial}\vm{f}^{i}\|^2.
\end{equation}
\end{corollary}

The following technical lemma is useful in deriving the uniform in time $H^1(\Omega)$ bound.

\begin{lemma}\label{lem:an}
Let $\{a_n\}$ be a nonnegative sequence that satisfies
$$
  a_{n+1}\le c_1\Delta t (a_n+a_{n-1}) + c_2\lambda^n + c_3 \quad\mbox{for $n=1,2,\ldots$,}
$$
where $c_i$, $i=1,2,3$, are  positive  numbers and $0<\lambda<1$.
Moreover, if $\Delta t< \frac{2\lambda}{(1+\sqrt{5})c_1}$, then,
\begin{equation}\label{lem:an:res}
    a_{n+1}\le \frac{c_3}{1-\frac{1+\sqrt{5}}{2}c_1\Delta t}+
    \lambda^n\Big(\frac{c_2}{1-\frac{1+\sqrt{5}}{2\lambda}c_1\Delta t}+a_1+\frac{\sqrt{5}-1}{2}a_0\Big).
\end{equation}
\end{lemma}

\begin{proof} Define  $b_{n+1}=a_{n+1}+\frac{\sqrt{5}-1}{2}c_1\Delta ta_n$. Then,
$$
 b_{n+1}\le \frac{1+\sqrt{5}}{2}c_1\Delta tb_n+c_2\lambda^n+c_3.
$$
A simple induction leads to
$$
   b_{n+1}\le \sum_{i=1}^n\Big(\frac{1+\sqrt{5}}{2}c_1\Delta t\Big)^{n-i}(c_2\lambda^i +c_3)+  \Big(\frac{1+\sqrt{5}}{2}c_1\Delta t\Big)^{n}b_1.
$$
Now if $\Delta t< \frac{2\lambda}{(1+\sqrt{5})c_1}$,  we have
$$
   b_{n+1}\le\frac{c_3}{1-\frac{1+\sqrt{5}}{2}c_1\Delta t}+ \frac{c_2\lambda^n}{1-\frac{1+\sqrt{5}}{2\lambda}c_1\Delta t}+\lambda^nb_1.
$$
The desired bound on $a_{n+1}$ follows from this inequality, the definition of $b_{n+1}$ and $b_1$, and the fact that $c_1\Delta t < 1$ under the assumption.
\end{proof}

\begin{rmk}If $\Delta t< \frac{\lambda}{(1+\sqrt{5})c_1}$, then \eqref{lem:an:res} implies
$$
 a_{n+1}\le 2c_3+\lambda^n(2c_2+a_1+\frac{\sqrt{5}-1}{2}a_0).
$$
\end{rmk}

\begin{theorem}\label{BDF:H1:stable}
The {\em BDF2} scheme \eqref{schemeBDF} is asymptotically stable with respect to the $H^1(\Omega)$ norm in the sense that the $H^1(\Omega)$ norm of the solution is uniformly bounded in time.
\end{theorem}

\begin{proof}
Set $\vm{v}=\vm{u}^{n+1}$ in the BDF scheme \eqref{schemeBDF} and use the skew-symmetry property of $a_{\Gamma}$ to obtain
$$
   a(\vm{u}^{n+1},\vm{u}^{n+1})=-\frac{1}{\Delta t}\big\langle\big\langle D\vm{u}^{n+1},\vm{u}^{n+1}\big\rangle\big\rangle
   + \widetilde{a}_{\Gamma}(\delta \vm{u}^{n+1}, \vm{u}^{n+1})
   +\big\langle\big\langle\big\langle\vm{f}^{n+1},\vm{u}^{n+1}\big\rangle\big\rangle\big\rangle.
$$
Note that
$$
-\frac{1}{\Delta t}\big\langle\big\langle D\vm{u}^{n+1},\vm{u}^{n+1}\big\rangle\big\rangle
+\big\langle\big\langle\big\langle \vm{f}^{n+1},\vm{u}^{n+1}\big\rangle\big\rangle\big\rangle\le
C\left(\|\frac{1}{\Delta t} D\vm{u}^{n+1}\|+\|\vm{f}^{n+1}\|\right)
\|\nabla \vm{u}^{n+1}\|
$$
and
$$
 \widetilde{a}_{\Gamma}(\delta \vm{u}^{n+1}, \vm{u}^{n+1})\le  \left(C\|\frac{1}{\Delta t}\delta \vm{u}^{n+1}\|+
 \frac{\Delta t}{4}\|\nabla\delta \vm{u}^{n+1}\|\right)\|\nabla  \vm{u}^{n+1}\|,.
$$
Using the coercivity condition \eqref{ineq2} and \eqref{BDF:Dunorm}, we deduce
$$
\begin{aligned}
  \|\nabla\vm{u}^{n+1}\| &\le C_9\left(\|\frac{1}{\Delta t} D\vm{u}^{n+1}\|+\|\frac{1}{\Delta t}\delta \vm{u}^{n+1}\|+\|\vm{f}^{n+1}\|
  + \Delta t(\|\nabla\vm{u}^{n}\|+\|\nabla\vm{u}^{n-1}\|)\right)\\
     &\le  C_{15} \lambda_1^{\frac{n}{2}} +C_{15}\max_i(\|\bar{\partial}\vm{f}^{i+1}\|+\|\vm{f}^{i+1}\|)+C_{14} \Delta t(\|\nabla\vm{u}^{n}\|+\|\nabla\vm{u}^{n-1}\|),
\end{aligned}
$$
provided that $\Delta t \le C_a$. The desired uniform in time estimate then follows from Lemma~\ref{lem:an} with $a_n= \|\nabla\vm{u}^{n}\|$, $c_1=C_{14}$, $c_2=C_{15}$, and $\lambda=\lambda_1^{\half}= (1+C_7\Delta t)^{-\half}$. Specifially, provided that the time step is small enough in the sense that $\Delta t\le\frac{\sqrt{1+\frac{4}{(1+\sqrt{5})C_{14}}}-1}{2C_7}$, the time-step condition in Lemma~\ref{lem:an} is verified, i.e.,
$
 \frac{(1+\sqrt{5})C_{14}} {2\lambda_1}\Delta t\le \half .
$
Hence  by Lemma~\ref{lem:an},
$$
    \|\nabla\vm{u}^{n+1}\|\le C\lambda_1^{\frac{n}{2}}+ 2C_{15}\max_i(\|\bar{\partial}\vm{f}^{i+1}\|+\|\vm{f}^{i+1}\|).
$$
\end{proof}

\subsection{Uniform in time $H^1(\Omega)$ bound for the AMB2 scheme}\label{sec:42}

In this subsection, we assume that the time-step restriction \eqref{timestepAMB} holds.

Utilizing the same arguments as for Lemma~\ref{lem:BDF:barpartialu}, we can deduce that the discrete time derivative of the solution of \eqref{schemeAMB} is uniformly bounded in time.

\begin{lemma}
For the {\em ABM2 scheme}, we have
$$
   \|\bar{\partial}\vm{u}^{n+1}\|^2\le C  \lambda_3^{n-1}  +C\max_i\|\bar{\partial}\vm{f}^{i+\half}\|^2,
$$
where the positive parameter $\lambda_3<1$ is defined in Theorem~{\em\ref{Longstab:AMB2}}.
\end{lemma}

\begin{lemma} \label{lem:an2}
Let $a_n$ be a nonnegative sequence and let
$$
  a_{n+1}\le c_4 a_n+ c_5a_{n-1} +  c_6 \lambda^{n-1}  +   c_7 \quad\mbox{for $n=1,2,\ldots$},
$$
where $c_i$, $i=4,\ldots,7$, are positive real numbers and $0<\lambda<1$. Let $\xi_1=\frac{\sqrt{c_4^2+4c_5}+c_4}{2}$ and $\xi_2=\frac{\sqrt{c_4^2+4c_5}-c_4}{2}$. If $c_4+c_5<1$, then
\begin{equation}\label{lem:an:res2}
    a_{n+1}\le \xi_1^n(a_1+\xi_2a_0)+\frac{c_6(\max(\lambda,\xi_1))^{n-1}}{1-\min\left(\frac{\lambda}{\xi_1},\frac{\xi_1}{\lambda}\right)}+\frac{c_7}{1-\xi_1}.
\end{equation}
\end{lemma}

\begin{proof}
Note that $\xi_1<1$ because $c_4+c_5<1$. Letting $b_{n+1}=a_{n+1}+\xi_2a_n$
we have
$$
  b_{n+1}\le \xi_1b_{n}+   c_6 \lambda^{n-1} + c_7.
$$
By induction,
$$
  b_{n+1}\le \xi_1^nb_1+ c_6\sum_{i=1}^{n-1}\xi_1^{n-i}\lambda^i + c_7 \sum_{i=0}^{n-1}\xi_1^i.
$$
Because $\xi_1<1$, then $\sum_{i=0}^{n-1}\xi_1^i \le \frac{1}{1-\xi_1}$ and
$$
  \sum_{i=1}^n\xi_1^{n-i}\lambda^i\le \frac{(\max(\lambda,\xi_1))^{n-1}}{1-\min\left(\frac{\lambda}{\xi_1},\frac{\xi_1}{\lambda}\right)}.
$$
Now \eqref{lem:an:res2} can be obtained.
\end{proof}

\begin{thm}
The {\em ABM2} scheme is   asymptotically stable with respect to the $H^1(\Omega)$ norm.
\end{thm}

\begin{proof} From \eqref{AMB2:scheme2},
$$
   a(D_\alpha\vm{u}^{n+1},\vm{u}^{n+1})=\big\langle\big\langle\big\langle\vm{f}^{n+\half},\vm{u}^{n+1}\big\rangle\big\rangle\big\rangle +{\mathcal E}_\Gamma - \big\langle\big\langle \bar{\partial}\vm{u}^{n+1},\vm{u}^{n+1}\big\rangle\big\rangle,
$$
where ${\mathcal E}_{\Gamma}$ is defined in Lemma~\ref{AMB2:Egamma}. Note that
$$
 a(D_\alpha\vm{u}^{n+1},\vm{u}^{n+1})\ge \beta_2\|\vm{u}^{n+1}\|_a^2+C_a(\beta_1-\beta_2) \|\nabla\vm{u}^{n+1}\|^2
 -\alpha_1 \|\vm{u}^{n}\|_a^2-\alpha_2 \|\vm{u}^{n-1}\|_a^2,
$$
and
$$
\big\langle\big\langle\big\langle\vm{f}^{n+\frac{1}{2}},\vm{u}^{n+1}\big\rangle\big\rangle\big\rangle
 - \big\langle\big\langle  \bar{\partial}\vm{u}^{n+1},\vm{u}^{n+1}\big\rangle\big\rangle \le
 C\left(\|\bar{\partial}\vm{u}^{n+1}\|^2+\|\vm{f}^{n+\frac{1}{2}}\|^2\right)+
\frac{C_a(\beta_1-\beta_2)}{9}\|\nabla \vm{u}^{n+1}\|^2.
$$
By following the proof in \eqref{AMB2:E} with small modifications, we obtain
$$
\begin{aligned}
  {\mathcal E}_{\Gamma} &\le C\Delta t^\half \|\nabla\vm{u}^{n+1}\|\sum_{j=n}^{n+1}\|\bar{\partial}\vm{u}^{j}\|^{1/2}(\|\nabla\vm{u}^{j}\|^{1/2}
  +\|\nabla\vm{u}^{j-1}\|^{1/2})\\
  &\le C  \sum_{j=n}^{n+1}\|\bar{\partial}\vm{u}^{j}\|^2+\frac{C_a(\beta_1-\beta_2)\Delta t^2}{9}(4\|\nabla\vm{u}^{n+1}\|^2+
  2\|\nabla\vm{u}^{n}\|^2+\|\nabla\vm{u}^{n-1}\|^2).
\end{aligned}
$$
Define $\widetilde{\alpha}_1=\alpha_1+\frac{2(\beta_1-\beta_2)}{9}\Delta t^2$ and $\widetilde{\alpha}_2=\alpha_2+\frac{(\beta_1-\beta_2)}{9}\Delta t^2$. Then,
$$
  \beta_2\|\vm{u}^{n+1}\|_a^2\le \widetilde{\alpha}_1\|\vm{u}^{n}\|_a^2+ \widetilde{\alpha}_2\|\vm{u}^{n-1}\|_a^2+C\lambda_3^{n-2}
  + C\max_i(\|\bar{\partial}\vm{f}^{i+\half}\|^2+\|\vm{f}^{i+\half}\|^2).
$$
Now, if $\Delta t$ is satisfies
$$
 \Delta t< \sqrt{\frac{3(\beta_2-\beta_3)}{\beta_1-\beta_2}}
$$
then  $\beta_2>\widetilde{\alpha}_1+\widetilde{\alpha}_2$ and then, by Lemma~\ref{lem:an2},
$$
\begin{aligned}
    \|\vm{u}^{n+1}\|_a^2&\le \xi_1^n(\|\vm{u}^{1}\|_a^2+\xi_2\|\vm{u}^{0}\|_a^2)+
    \frac{C(\max(\lambda_3,\xi_1))^n}{\beta_2\lambda_3\left(1-\min\left(\frac{\lambda_3}{\xi_1},\frac{\xi_1}{\lambda_3}\right)\right)}
\\&\qquad\qquad + \frac{ C\max_i(\|\bar{\partial}\vm{f}^{i+\half}\|^2+\|\vm{f}^{i+\half}\|^2)}{\beta_2(1-\xi_1)},
\end{aligned}
$$
where   $\xi_1=\frac{\sqrt{\widetilde{\alpha}_1^2+4\widetilde{\alpha}_2\beta_2}+\widetilde{\alpha}_1}{2\beta_2}$
and $\frac{\sqrt{\widetilde{\alpha}_1^2+4\widetilde{\alpha}_2\beta_2}-\widetilde{\alpha}_1}{2\beta_2}$.
\end{proof}

\section{Error analysis}\label{sec:5}

In this section, we study the convergence of the fully-discrete BDF2 scheme, where spatial discretization is effected using finite element methods. A similar study yielding similar results can be done for the AMB2 scheme; however, for the sake of brevity, we omit that study.

Let $\mathbf{H}_{f,h}\subset \mathbf{H}_f$, ${H}_{p,h}\subset {H}_p$, and ${Q}_h\subset {Q}$ denote conforming finite element spaces. Let $\mathbf{W}_h=\mathbf{H}_{f,h}\times {H}_{p,h}$. We assume that the mesh is regular and that the parameter $h$ is a measure of the grid size. We use continuous piecewise polynomials of degrees $k$, $k$, and $k-1$ for the spaces $\mathbf{H}_{f,h}$, ${H}_{p,h}$, and ${Q}_h$, respectively. See \cite{Ciarlet} for details concerning such finite element discretizations. We also assume that the fluid velocity and pressure spaces $\mathbf{H}_{f,h}$ and ${Q}_h$ satisfy the discrete inf-sup condition necessary for ensuring the stability of the finite element discretization; see \cite{GR86}.

\begin{definition}
The Stokes-Darcy projection $P_h: \mathbf{W}\times {Q}\rightarrow  \mathbf{W}_h\times {Q}_h$ is defined as follows. For any $\vm{u}\in \mathbf{W}$ and $p\in {Q}$, let  $P_h \vm{u}\in \mathbf{W}_h$ and $P_h p\in {Q}_h$ denote the finite element solution of
$$
\begin{aligned}
   a(P_h\vm{u}, \vm{v}_h) + b(\mathbf{v}_h,P_hp) + a_{\Gamma}(P_h\vm{u}, \vm{v}_h)&=a(\vm{u}, \vm{v}_h) + b(\mathbf{v}_h,p) + a_{\Gamma}(\vm{u}, \vm{v}_h)
   \\b(P_h\mathbf{u},q_h)&=b(\mathbf{u},q_h)
\end{aligned}
$$
for all $\vm{v}_h\in \mathbf{W}_h$ and $q_h\in Q_h$.
 \end{definition}

It is easy to see that for any  $\vm{u}\in \mathbf{W}$ and $p\in {Q}$, the exist unique $P_h \vm{u}\in \mathbf{W}_h$ and $P_h p\in {Q}_h$.
Moreover, if we assume that $\vm{u}\in (H^{k+1}(\Omega_f))^d\times H^{k+1}(\Omega_p)$ and $p\in H^{k}(\Omega_f)$, then (see, e.g., \cite{Cao2010b}),
\begin{equation}\label{H1Ph}
  \|\vm{u}-P_h\vm{u}\|+ h\|\nabla(\vm{u}-P_h\vm{u})\|\le h^{k+1}(\|\vm{u}\|_{H^{k+1}}+\|p\|_{H^k}).
\end{equation}

\begin{rmk}
The estimate \eqref{H1Ph} and the optimal error estimates derived below assume that $\vm{u}\in (H^{k+1})^d\times H^{k+1}(\Omega_p)$ which requires that the interface $\Gamma$ be sufficiently smooth. In this case, the finite elements may need to be modified near the interface, e.g., by using isoparametric finite element approximations {\em\cite{Ciarlet}.} In any case, in this paper we assume that the optimal error order of convergence can be obtained for the steady-state Stokes-Darcy problems using the same grids and finite element spaces.
\end{rmk}

\subsection{Error analysis of the BDF2-FEM scheme}\label{sec:51}

The fully discrete BDF2-FEM scheme is defined as follows: for $n=0,1,2\ldots$, seek $\vm{u}^{n+1}_h\in \mathbf{W}_h$ and $p^{n+1}_h\in {Q}_h$ such that
\begin{equation}\label{fem}
\begin{aligned}
&\frac{1}{\Delta
t}\big\langle\big\langle D\vm{u}^{n+1}_h,\vec{\mathbf{v}}_h \big\rangle\big\rangle +a({\vm{u}}^{n+1}_h,\vm{v}_h) +b(\mathbf{v}_h,p^{n+1}_h) +a_{st}(\vm{u}^{n+1}_h,\vec{\mathbf{v}}_h)\\
&\qquad\qquad
=\big\langle\big\langle\big\langle\vec{\mathbf{f}}^{n+1},{\vec{\mathbf{v}}}_h\big\rangle\big\rangle\big\rangle
-a_\Gamma(2\vm{u}^{n}_h-\vm{u}^{n-1}_h,\vec{\mathbf{v}}_h)+
a_{st}(2\vm{u}^{n}_h-\vm{u}^{n-1}_h,\vec{\mathbf{v}}_h) \\
&b(\mathbf{u}^{n+1}_h,q_h)=0
\end{aligned}
\end{equation}
are satisfied for all $\vm{v}_h\in \mathbf{W}_h$ and $q_h\in {Q}_h$. Note that for all $\vm{v}_h\in \mathbf{W}_h$ and $q_h\in {Q}_h$,  the exact solution satisfies
\begin{equation}\label{exacts}
\begin{aligned}
 \big\langle\big\langle \vm{u}_t,\vm{v}_h\big\rangle\big\rangle  + a(\vm{u},\vm{v}_h)+a_{\Gamma}(\vm{u},\vm{v}_h)+b(\mathbf{v}_h,p)&=\big\langle\big\langle\big\langle \vm{f},\vm{v}_h\big\rangle\big\rangle\big\rangle \\
  b(\mathbf{u},q_h)&=0.
\end{aligned}
\end{equation}

\begin{theorem}\label{thm:BDF:L2} Assume that the solution of the Darcy-Stokes problem \eqref{SDsys} is sufficiently regular in the sense that
$$
   \vm{u} \in H^3(0,T;H^1)\cap H^2(0,T;H^{k+1}),
$$
that the time-step restrictions \eqref{timestepBDF} and \eqref{timestepBDF2} are satisfied, and that the finite element spaces are chosen so that the projection error bound \eqref{H1Ph} holds.
Then, the solution of the fully-discrete {\em BDF2} scheme \eqref{fem} satisfies the error estimate
$$
\begin{aligned}
   \|\vm{u}(t)-\vm{u}^n_h\|^2&\le \|P_h\vm{u}(t_0)-\vm{u}^0_h\|^2+ \|P_h\vm{u}(t_1)-\vm{u}^1_h\|^2 + C\Delta t(\|\nabla (P_h\vm{u}(t_0)-\vm{u}^0_h)\|^2\\
  &\qquad+ \|\nabla (P_h\vm{u}(t_1)-\vm{u}^1_h)\|^2)
  +C(\Delta t^4+h^{2(k+1)}).
\end{aligned}
$$
Moreover, if the solution of the Stokes-Darcy problem \eqref{SDsys} is long-time regular in the sense that
$$
   \vm{u} \in W^{3, \infty}(0,\infty;H^1)\cap W^{2, \infty}(0,\infty;H^{k+1}),
$$
then, there exists a constant $C_a$ and a generic constant $C$ independent of $\Delta t, h$, or $n$ such that the solution of the {\em BDF2} scheme \eqref{fem}  satisfies the {\em uniform in time error estimates}
\begin{equation}\label{BDF:L2:bound}
\begin{aligned}
   \|\vm{u}(t_n)-\vm{u}^n_h\|^2&\le C \lambda_1^{n-2} (\|P_h\vm{u}(t_0)-\vm{u}^0_h\|^2+ \|P_h\vm{u}(t_1)-\vm{u}^1_h\|^2) \\
  &+ C\Delta t^2 \lambda_1^{n-2} (\|\nabla (P_h\vm{u}(t_0)-\vm{u}^0_h)\|^2+ \|\nabla (P_h\vm{u}(t_1)-\vm{u}^1_h)\|^2) \\
  &+C(\Delta t^4+h^{2(k+1)} )\quad  \forall\, n
\end{aligned}
\end{equation}
and
$$
\begin{aligned}
    &\|\nabla(\vm{u}(t_{n+1})-\vm{u}^{n+1}_h)\|^2 +  \|p(t_{n+1})-p_h^{n+1}\|^2\\\
    &\quad\le C\lambda_1^{n-2}(\|\bar\partial P_h\vm{u}(t_1)-\bar\partial\vm{u}^1_h\|^2+ \|\bar\partial P_h\vm{u}(t_2)-\bar\partial\vm{u}^2_h\|^2) \\
  &\qquad+ C\Delta t^2\lambda_1^{n-2} (\|\nabla (\bar\partial P_h\vm{u}(t_1)-\bar\partial\vm{u}^1_h)\|^2+ \|\nabla (\bar\partial P_h\vm{u}(t_2)-\bar\partial\vm{u}^2_h)\|^2) \\
  &\qquad+C(\Delta t^2+h^{2k})\quad \forall\, n
\end{aligned}
$$
  provided that $\Delta t\le C_a$.
  \end{theorem}

\begin{proof}
Let $\vm{e}^n=\vm{u}(t_n)-\vm{u}^n_h$ denote the error at the time $t=t_n$. Then, from \eqref{fem} and \eqref{exacts}, we have
\begin{align}
 &
 \begin{aligned}
 &\frac{1}{\Delta t}\big\langle\big\langle  D\vm{e}^{n+1}, \vm{v}_h\big\rangle\big\rangle  +a(\vm{e}^{n+1},\vm{v}_h)+b(\mathbf{v}_h, p(t_{n+1})-p_h^{n+1}) + a_{\Gamma}(\vm{e}^{n+1}, \vm{v}_h)\\
 &\qquad\qquad -\widetilde{a}_{\Gamma}(\delta \vm{e}^{n+1},\vm{v}_h)
 =\big\langle \big\langle  \omega_1^{n+1}, \vm{v}_h\big\rangle\big\rangle   -\widetilde{a}_{\Gamma}(\delta \vm{u}(t_{n+1}),\vm{v}_h)
\end{aligned}  \label{error:BDFe1} \\
  &b(\mathbf{e}^{n+1}, q_h)=0,  \nonumber
\end{align}
where ${\boldsymbol\omega}_1^{n+1}=-\vm{u}_t(t_{n+1})+\frac{1}{\Delta t}D\vm{u}(t_{n+1})$.
Let $\vm{\boldsymbol\rho}^n=\vm{u}(t_n)-P_h \vm{u}(t_n)$ and $\vm{\boldsymbol\theta}^n=P_h\vm{u}(t_n)-\vm{u}^n_h$. Then, $\vm{\boldsymbol\theta}^n\in \mathbf{W}_h$ and is discretely divergence free, i.e.,
\begin{equation}\label{disdf}
   b({\boldsymbol\theta}^n, q_h)=0 \quad \forall\, q_h\in {Q}_h.
\end{equation}
Because $\vm{e}^n=\vm{\boldsymbol\theta}^n+\vm{\boldsymbol\rho}^n$, the error equation \eqref{error:BDFe1} can be recast as
$$
\begin{aligned}
  &\frac{1}{\Delta t}\big\langle\big\langle  D\vm{{\boldsymbol\theta}}^{n+1}, \vm{v}_h\big\rangle\big\rangle  +a(\vm{{\boldsymbol\theta}}^{n+1},\vm{v}_h)+ a_{\Gamma}(\vm{{\boldsymbol\theta}}^{n+1}, \vm{v}_h)-\widetilde{a}_{\Gamma}(\delta \vm{{\boldsymbol\theta}}^{n+1},\vm{v}_h) \\
  &\quad=  \big\langle\big\langle  {\boldsymbol\omega}_1^{n+1}, \vm{v}_h\big\rangle\big\rangle
  -\widetilde{a}_{\Gamma}(\delta \vm{u}(t_{n+1}),\vm{v}_h)-b(\mathbf{v}_h, p(t_{n+1})-p_h^{n+1}) \\
  &\qquad -\frac{1}{\Delta t}\big\langle\big\langle  D\vm{{\boldsymbol \rho}}^{n+1},\vm{v}_h\big\rangle\big\rangle - a(\vm{{\boldsymbol \rho}}^{n+1},\vm{v}_h)-a_{\Gamma}(\vm{{\boldsymbol \rho}}^{n+1},\vm{v}_h)
  + \widetilde{a}_{\Gamma}(\delta \vm{{\boldsymbol \rho}}^{n+1}, \vm{v}_h) \\
    &\quad=  \big\langle\big\langle  {\boldsymbol\omega}_1^{n+1}, \vm{v}_h\big\rangle\big\rangle
  -\widetilde{a}_{\Gamma}(\delta \vm{u}(t_{n+1}),\vm{v}_h)-b(\mathbf{v}_h, P_hp(t_{n+1})-p_h^{n+1}) \\
  &\qquad -\frac{1}{\Delta t}\big\langle\big\langle  D\vm{{\boldsymbol \rho}}^{n+1},\vm{v}_h\big\rangle\big\rangle
  +\widetilde{a}_{\Gamma}(\delta \vm{{\boldsymbol \rho}}^{n+1}, \vm{v}_h).
\end{aligned}
$$
Setting $\vm{v}_h=\vm{\boldsymbol\theta}^{n+1}$, noting that $a_{\Gamma}(\vm{\boldsymbol\theta}^{n+1}, \vm{\boldsymbol\theta}^{n+1})=0$, and using \eqref{disdf} results in
\begin{equation}\label{BDF:error:L2}
\begin{aligned}
  &\frac{1}{\Delta t}\big\langle\big\langle D\vm{{\boldsymbol \theta}}^{n+1}, \vm{{\boldsymbol \theta}}^{n+1}\big\rangle\big\rangle
    +a(\vm{{\boldsymbol \theta}}^{n+1},\vm{{\boldsymbol \theta}}^{n+1})
  -\widetilde{a}_{\Gamma}(\delta \vm{{\boldsymbol \theta}}^{n+1},\vm{{\boldsymbol \theta}}^{n+1}) =\big\langle\big\langle  {\boldsymbol\omega}_1^{n+1}, \vm{{\boldsymbol \theta}}^{n+1}\big\rangle\big\rangle  \\
&
  \qquad\qquad -\widetilde{a}_{\Gamma}(\delta \vm{u}(t_{n+1}),\vm{{\boldsymbol \theta}}^{n+1})   -\frac{1}{\Delta t}\big\langle\big\langle  D\vm{{\boldsymbol \rho}}^{n+1},\vm{{\boldsymbol \theta}}^{n+1}\big\rangle\big\rangle
  +\widetilde{a}_{\Gamma}(\delta \vm{{\boldsymbol \rho}}^{n+1},\vm{{\boldsymbol \theta}}^{n+1}).
\end{aligned}
\end{equation}
Letting $\vm{w}_n=[\vm{{\boldsymbol \theta}}^{n+1}, \vm{{\boldsymbol \theta}}^n]^T$ and $E_n=|\vec{\mathbf{w}}_{n}|^2_G+\frac{C_a\Delta
t}{2} \|\nabla\vm{{\boldsymbol \theta}}^{n+1} \|^2+\frac{C_a\Delta
t}{8} \|\nabla\vm{{\boldsymbol \theta}}^{n} \|^2$ and following the lines of the proof of Theorem~\ref{Longstab:BDF2}, we have
\begin{equation}\label{error:En}
\begin{aligned}
& E_n +\frac{C_a}{2}\Delta
t\left\|\nabla\vm{{\boldsymbol \theta}}^{n+1}\right\|^2+\frac{C_a}{4}\Delta
t\left\|\nabla\vm{{\boldsymbol \theta}}^{n}\right\|^2 \le E_{n-1} \\
&\qquad   +C\Delta t( \|{\boldsymbol\omega}_1^{n+1}\|^2+ \|\nabla \delta \vm{u}(t_{n+1})\|^2+\|\frac{1}{\Delta t}D\vm{{\boldsymbol \rho}}^{n+1}\|^2+\|\nabla \delta \vm{{\boldsymbol \rho}}^{n+1}\|^2).
\end{aligned}
\end{equation}
By Taylor's theorem with the integral form of the remainder, we have
\begin{equation}\label{error:tr1}
   \|{\boldsymbol\omega}_1^{n+1}\|^2\le C\Delta t^3\int_{t_{n-1}}^{t_{n+1}}\|\vm{u}_{ttt}\|^2dt
   	\le C\Delta t^4 \|\vm{u}_{ttt}\|^2_{L^\infty(0, t_{n+1})}
\end{equation}
and
\begin{equation}\label{error:tr2}
\|\nabla \delta \vm{u}(t_{n+1})\|^2\le  C\Delta t^3\int_{t_{n-1}}^{t_{n+1}}\|\nabla \vm{u}_{ttt}\|^2dt
    \le C\Delta t^4 \|\nabla\vm{u}_{ttt}\|^2_{L^\infty(0, t_{n+1})}.
\end{equation}
Moreover, using \eqref{H1Ph}, we have
\begin{equation}\label{error:tr3}
\begin{aligned}
   \|\frac{1}{\Delta t}D\vm{\boldsymbol\rho}^{n+1}\|^2&\le Ch^{2(k+1)}\|\frac{1}{\Delta t}D\vm{u}(t_{n+1})\|_{H^{k+1}}^2
   \\&\le
    C\frac{h^{2(k+1)}}{\Delta t}\int_{t_{n-1}}^{t_{n+1}}\|\vm{u}_t\|_{H^{k+1}}^2dt
    \le C h^{2(k+1)}\|\vm{u}_t\|_{L^\infty(0, t_{n+1}; H^{k+1})}^2
\end{aligned}
\end{equation}
and
\begin{equation}\label{error:tr4}
\begin{aligned}
  \|\nabla \delta \vm{\boldsymbol\rho}^{n+1}\|^2\le Ch^{2k}\|\delta  \vm{u}(t_{n+1})\|_{H^{k+1}}^2 &\le Ch^{2k}\Delta t^3\int_{t_{n-1}}^{t_{n+1}}\|\vm{u}_{tt}\|_{H^{k+1}}^2dt\\
  &\le C h^{2k} \Delta t^4 \|\vm{u}_{tt}\|_{L^\infty(0, t_{n+1}; H^{k+1})}^2 .
\end{aligned}
\end{equation}
Combining \eqref{error:En}--\eqref{error:tr4}, we have
\begin{equation}
\begin{aligned}
&E_n+\frac{C_a}{2}\Delta
t\sum_{i=1}^{n+1}\left\|\nabla\vm{\theta}^{i}\right\|^2\le E_0+  C\Big(\Delta t^4\int_{0}^{t_{n+1}}(\|\vm{u}_{ttt}\|^2 + \|\nabla \vm{u}_{ttt}\|^2)dt \\
&\qquad  +
{h^{2(k+1)}}\int_{0}^{t_{n+1}}\|\vm{u}_t\|_{H^{k+1}}^2dt + h^{2k}\Delta t^4\int_{0}^{t_{n+1}}\|\vm{u}_{tt}\|_{H^{k+1}}^2dt\Big).
\end{aligned}
\end{equation}
The desired finite time error estimate follows from this and the assumed bound on the projection error $\vm{{\boldsymbol \rho}}^n$.

For the uniform in time $L^2(\Omega)$ bound, we again use \eqref{error:En}--\eqref{error:tr4} to obtain
\begin{equation}\label{eqq1}
\begin{aligned}
&E_n +\frac{C_a}{2}\Delta
t\left\|\nabla\vm{{\boldsymbol \theta}}^{n+1}\right\|^2+\frac{C_a}{4}\Delta
t\left\|\nabla\vm{{\boldsymbol \theta}}^{n}\right\|^2\\
&\quad\le  E_{n-1}+C\Delta t\Big(\Delta t^4\|\vm{u}_{ttt}\|^2_{L^\infty(0, \infty)}+\Delta t^4 \|\nabla\vm{u}_{ttt}\|^2_{L^\infty(0, \infty)} \\
&\qquad +h^{2(k+1)}\|\vm{u}_{t}\|^2_{L^\infty(0, \infty;H^{k+1})}
+h^{2k}\Delta t^4\|\vm{u}_{tt}\|^2_{L^\infty(0, \infty;H^{k+1})}\Big)\\
&\quad\le  E_{n-1}+C\Delta t(\Delta t^4+h^{2(k+1)}).
\end{aligned}
\end{equation}
Using the Poincar\'{e} inequality and the definition of the $G$-norm, we have
$$
\begin{aligned}
&\frac{C_a}{2}\Delta
t\left\|\nabla\vm{{\boldsymbol \theta}}^{n+1}\right\|^2+\frac{C_a}{4}\Delta
t\left\|\nabla\vm{{\boldsymbol \theta}}^{n}\right\|^2\\
 &\quad\ge\frac{C_a}{4}\Delta
t\left\|\nabla\vm{{\boldsymbol \theta}}^{n+1}\right\|^2+\frac{C_a}{8}\Delta
t\left\|\nabla\vm{{\boldsymbol \theta}}^{n}\right\|^2+\frac{C^2_l C_a}{8C^2_P}|\vm{w}_n|^2_G.
\end{aligned}
$$
Then, with $C_7$ defined as in Theorem \ref{Longstab:BDF2}, we have from \eqref{eqq1},
$$
(1+C_7\Delta t)E_n\le E_{n-1}+C\Delta t(\Delta t^4+h^{2(k+1)})
$$
A simple induction argument then leads to
$$
\begin{aligned}
E_n &\le \Big(\frac{1}{1+C_7\Delta t}\Big)^{n-2} E_2+C(\Delta t^4+h^{2(k+1)})\\
&\le C\lambda^{n-2}_1(\|\vm{{\boldsymbol \theta}}^{1}\|^2+\|\vm{{\boldsymbol \theta}}^{0}\|^2)+C\lambda^n_1\Delta t(\|\nabla\vm{{\boldsymbol \theta}}^{1}\|^2+\|\nabla\vm{{\boldsymbol \theta}}^{0}\|^2)+C(\Delta t^4+h^{2(k+1)}),
\end{aligned}
$$
where $\lambda_1$ is defined as in Theorem~\ref{Longstab:BDF2}. The bound \eqref{BDF:L2:bound} then follows from Corollary~\ref{Longstab:BDF2}.

The uniform in time $H^1(\Omega)$-norm error estimate on the velocity and the $L^2(\Omega)$ error estimate on the pressure can be derived as well after we combine the technique used above with techniques from \S\ref{sec:4}.
Indeed, from \eqref{error:tr1} and \eqref{error:tr2} and using the triangle inequality, we have
\begin{equation}\label{error:tr1:bp}
   \|\bar\partial {\boldsymbol\omega}_1^{n+1}\|^2\le C\Delta t\int_{t_{n-2}}^{t_{n+1}}\|\vm{u}_{ttt}\|^2dt\le C\Delta t^2\|\vm{u}_{ttt}\|^2_{L^{\infty}(0,t_{n+1})}
\end{equation}
and
\begin{equation}\label{error:tr2:bp}
\quad \|\nabla \bar\partial \delta \vm{u}(t_{n+1})\|^2\le  C\Delta t\int_{t_{n-2}}^{t_{n+1}}\|\nabla \vm{u}_{ttt}\|^2dt\le C\Delta t^2 \|\nabla \vm{u}_{ttt}\|^2_{L^\infty(0,t_{n+1})}.
\end{equation}
Moreover, by the definitions of  $P_h$, $\bar\partial$, and $D$,
\begin{equation}\label{error:tr3:bp}
\begin{aligned}
  & \|\frac{1}{\Delta t}\bar\partial D\vm{{\boldsymbol \rho}}^{n+1}\|^2\le Ch^{2(k+1)}\|\frac{1}{\Delta t}\bar\partial D\vm{u}(t_{n+1})\|_{H^{k+1}}^2
  \\&\qquad\le
    C\frac{h^{2(k+1)}}{\Delta t}\int_{t_{n-2}}^{t_{n+1}}\|\vm{u}_{tt}\|_{H^{k+1}}^2dt
    \le Ch^{2(k+1)} \|\vm{u}_{tt}\|_{L^{\infty}(0,t_{n+1};H^{k+1})}^2
\end{aligned}
\end{equation}
and by  the triangle inequality and \eqref{error:tr4},
\begin{equation}\label{error:tr4:bp}
  \|\nabla \bar\partial\delta \vm{{\boldsymbol \rho}}^{n+1}\|^2 \le C h^{2k} \Delta t^2 \|\vm{u}_{tt}\|_{L^\infty(0, t_{n+1}; H^{k+1})}^2.
\end{equation}
We combine \eqref{error:tr1:bp}--\eqref{error:tr4:bp} with the stability proof of Lemma~\ref{lem:BDF:barpartialu}, with small modification for the initial steps; see Corollary~\ref{Longstab:BDF2-1}. As a result, we obtain
\begin{equation}\label{eq555}
\begin{aligned}
  & \|\bar\partial \vm{{\boldsymbol \theta}}^{n+1}\|^2\le C\lambda_1^{n-2} (\|\bar\partial\vm{{\boldsymbol \theta}}^1\|^2+ \|\bar\partial\vm{{\boldsymbol \theta}}^2\|^2)
\\&\qquad  + C\Delta t^2\lambda_1^{n-2}(\|\nabla  \bar\partial P_h\vm{{\boldsymbol \theta}}^1\|^2+ \|\nabla  \bar\partial P_h\vm{{\boldsymbol \theta}}^2\|^2)
  +C(\Delta t^2+h^{2(k+1)}).
\end{aligned}
\end{equation}
Note that
\begin{equation}\label{error:tr4:2}
   \|\frac{\delta \vm{u}(t_{i})}{\Delta t}\|^2\le
      C\Delta t^2 \|\vm{u}_{tt}\|_{L^{\infty}(0,t_i)}^2,
\quad
  \|\frac{\delta {\boldsymbol \rho}^{i}}{\Delta t}\|^2\le
      C\Delta t^2h^{2(k+1)} \|\vm{u}_{tt}\|_{L^{\infty}(0,t_i;H^{k+1})}^2.
\end{equation}
Combining \eqref{eq555} and \eqref{error:tr4:2} with \eqref{error:tr1} and \eqref{error:tr3} and following the proof of  Theorem~\ref{BDF:H1:stable},
we have from \eqref{BDF:error:L2}
$$
\begin{aligned}
& \|\nabla \vm{{\boldsymbol \theta}}^{n+1}\|^2\le C\lambda_1^{n-2} (\|\bar\partial\vm{{\boldsymbol \theta}}^1\|^2+ \|\bar\partial\vm{{\boldsymbol \theta}}^2\|^2)
 \\&\qquad + C\Delta t^2\lambda_1^{n-2}(\|\nabla  \bar\partial P_h\vm{{\boldsymbol \theta}}^1\|^2+ \|\nabla  \bar\partial P_h\vm{{\boldsymbol \theta}}^2\|^2)
 +C(\Delta t^2+h^{2(k+1)}).
\end{aligned}
$$
After adding the estimate of $\|\nabla \vm{{\boldsymbol \rho}}^{n+1}\|$ (see \eqref{H1Ph}), we obtain the bound for $\|\nabla(\vm{u}(t_{n+1})-\vm{u}^{n+1}_h)\|^2 $.

The error estimate for the pressure $\|p-p_h\|$ can be obtained by standard mixed finite element analyses; see \cite{GR86}.
\end{proof}

\begin{rmk}
The uniform in time estimates given above imply that the method can be used to obtain an approximate solution of the steady-state equations in case the forcing term is time independent. This follows because the truncation errors listed in \eqref{error:En}--\eqref{error:tr4} vanish for the time-independent problem. Consequently, we have
$$  \|\vm{u}_h^n-\vm{u}^n\|^2 \le C \lambda_1^n (\|\vm{u}_h^0-\vm{u}^0\|^2
  + \Delta t \|\nabla (\vm{u}_h^0-\vm{u}^0)\|^2)
$$
for the steady-state problem.

In the steady-state case, the methods we study can be viewed as a domain decomposition method with the discrete time $n$ playing the role of an iteration number; see {\em\cite{Chen2011}} for a related scheme. The current scheme also enjoys an exponential rate of convergence as does the iterative scheme proposed in {\em\cite{Chen2011}.}
\end{rmk}

\begin{rmk}
Note that the uniform in time error estimate for the velocity with respect to the $H^1(\Omega)$ norm and the the uniform error estimate for the pressure are not second order in time. We do not know if this is an artifact of our approach. However, our numerical experiments in the next section suggest that the long-time convergence rate for the pressure approximation may very well be first order as the analysis suggests.
\end{rmk}

\section{Numerical results}\label{sec:6}

Using three numerical examples, we now illustrate the theoretical results of the previous section.

As was done in previous work \cite{Chen2011,Mu2010}, we set $\Omega_f=(0,1)\times(1,2)$ and $\Omega_p=(0,1)\times(0,1)$, with $\Omega_f$ and $\Omega_p$ separated by the interface $\Gamma=(0,1)\times\{1\}$. We choose the standard continuous piecewise-quadratic finite element space, defined with respect to the matrix domain $\Omega_p$, for approximating the hydraulic head $\phi$. We also choose the Hood-Taylor element pair, defined with respect to the conduit domain $\Omega_f$, i.e., continuous piecewise-quadratics and continuous piecewise-linear finite element spaces for the fluid velocity and pressure approximations, respectively. Uniform triangular meshes are created by first dividing the rectangular domains $\Omega_p$ and $\Omega_f$ into identical small squares and then dividing each square into two triangles. For illustrating the short-time properties of our schemes, we set the final time to $T=1$; for illustrating the long-time behavior, we set $T=100$.

We use three examples with exact solutions. Example 1 is taken from \cite{Mu2010}, Example 2 from \cite{Chen2011}, and Example 3 is a slight modification of an example in \cite{Cao2010b}. To illustrate the accuracy of our schemes, we assume that the error is of the order $O(h^{\theta_1}+\Delta t^{\theta_2})$. We set $\Delta t=h^\theta$ and quantify the numerically estimated order of convergence  $r_\theta=\min(\theta_1,\theta\theta_2)$ with respect to $h$ by calculating
$$
r_\theta\approx \log_2{\frac{\|u_{2h,\theta}-u_{exact}\|_{l^2}}{\|u_{h,\theta}-u_{exact}\|_{l^2}}}.
$$
Here, we use the discrete $l^2$ norm of nodal values to measure errors.

\vskip5pt
\noindent{\bf Example 1.}
We set the exact solution to \cite{Mu2010}
$$
\begin{aligned}
{\bf u}_f({\bf x},t)&=\Big([x^2(y-1)^2+y]\cos t\,\,,\,\,[-\frac{2}{3}x(y-1)^3+2-\pi\sin (\pi x)]\cos t\Big) \\
p_f({\bf x},t)&=[2-\pi\sin (\pi x)]\sin\left(\frac{\pi}{2}y\right)\cos t \\
\phi({\bf x},t)&=[2-\pi\sin (\pi x)][1-y-\cos (\pi y)]\cos t .
\end{aligned}
$$
The right-side data in the partial differential equations, initial conditions, and boundary conditions are then chosen correspondingly.
As done in \cite{Mu2010}, we set the parameters $\gamma_p=\gamma_f=g=S=\nu=\alpha_{BJ}=1$ and ${\mathbb K}={\mathbb I}$; also, we set $\alpha=0.8$ for the AMB2 scheme.

For Table \ref{Ex1:1}, we set $\Delta t=h$ and present results for several values of $h$; the results illustrate the second-order in time accuracy for $\phi$, ${\bf u}_f$, and $p_f$. We also notice that BDF2 has a significantly smaller error than AMB2, illustrating the advantage of the former over the latter scheme, at least for this example. For Tables \ref{Ex1:2} and \ref{Ex1:3}, $\Delta t$ is chosen to be a power of $h$ to illustrate the spatial convergence rates. The results in those tables indicate that the spatial accuracy seems higher than the third-order suggested by our analysis. The extra half order of accuracy may be attributed to super-convergence or super-approximation behaviors; see \cite{Chen2010} for a study of this phenomenon for the steady state case.

\begin{table}[h!]
\begin{center}
\begin{tabular}{|c||c|c|c|c|c|c|}
\cline{2-7}
\multicolumn{1}{c}{} & \multicolumn{2}{|c|}{$e_\phi$} & \multicolumn{2}{c|}{$e_{\bf u}$} & \multicolumn{2}{c|}{$e_p$} \\
\hline
$h$ &  BDF2&  AMB2&  BDF2&   AMB2&  BDF2&   AMB2\\
\hline 1/16 &5.76e-005&3.43e-003  &  8.26e-005  &  1.11e-004 & 1.15e-002 &4.11e-002\\
\hline 1/32 &9.53e-006  &  8.76e-004  &  1.98e-005  &  2.74e-005 & 3.02e-003 &1.07e-002\\
\hline 1/64 &2.35e-006  &  2.21e-004  &  4.85e-006  &  6.79e-006 & 7.73e-004&2.71e-003\\
\hline 1/128 &6.00e-007  &  5.55e-005 & 1.20e-006  &1.69e-006 &1.96e-004 &6.85e-004\\
\hline
\hline $r_{avg}$ & 2.20 & 1.98 & 2.04 &2.01 & 1.97 &1.97\\
\hline
\end{tabular}
\end{center}
\caption{Relative error and order of accuracy with respect to the spatial grid size $h$ for Example {\em1} at $t=1$ and with $\Delta t=h$.} \label{Ex1:1}
\end{table}

\begin{table}[h!]
\begin{center}
\begin{tabular}{|c||c|c|c|c|c|c|}
\cline{2-7}
\multicolumn{1}{c}{} & \multicolumn{2}{|c|}{$e_\phi$} & \multicolumn{2}{c|}{$e_{\bf u}$} & \multicolumn{2}{c|}{$e_p$} \\
\hline
$h$ &  BDF2&  AMB2&  BDF2&   AMB2&  BDF2&   AMB2\\
\hline 1/8  &6.16e-004  &6.83e-004    &  8.14e-005  &  8.37e-005 & 2.81e-002 &3.04e-002\\
\hline 1/16 &5.39e-005  &  6.46e-005  &  7.67e-006  &  7.86e-006 & 7.71e-003 &7.93e-003\\
\hline 1/32 &4.70e-006  &  6.01e-006  &  6.99e-007  &  7.16e-007 & 2.03e-003 &2.05e-003\\
\hline 1/64 &4.13e-007  &  5.51e-007  & 6.26e-008   &6.41e-008   & 5.22e-004 &5.24e-004\\
\hline
\hline $r_{avg}$ & 3.51 & 3.43 & 3.45 &3.45 & 1.92 &1.95\\
\hline
\end{tabular}
\end{center}
\caption{Same information as in Table {\rm\ref{Ex1:1}} but for $\Delta t=h^{3.5/2}$.} \label{Ex1:2}
\end{table}

\begin{table}[h!]
\begin{center}
\begin{tabular}{|c||c|c|c|c|c|c|}
\cline{2-7}
\multicolumn{1}{c}{} & \multicolumn{2}{|c|}{$e_\phi$} & \multicolumn{2}{c|}{$e_{\bf u}$} & \multicolumn{2}{c|}{$e_p$} \\
\hline
$h$ &  BDF2&  AMB2&  BDF2&   AMB2&  BDF2&   AMB2\\
\hline 1/8 &6.17e-004   &  5.82e-004  &  8.11e-005  &  8.17e-005 & 2.78e-002 &2.85e-002\\
\hline 1/16 &5.40e-005  &  5.21e-005  &  7.66e-006  &  7.69e-006 & 7.65e-003 &7.73e-003\\
\hline 1/32 &4.71e-006  &  4.62e-006  &  6.99e-007  &  7.01e-007 & 2.04e-003 &2.03e-003\\
\hline 1/64 &4.13e-007  &  4.09e-007  &  6.26e-008  &  6.28e-008 & 5.22e-004 &5.22e-004\\
\hline
\hline $r_{avg}$ & 3.52 & 3.49 & 3.45 &3.45 & 1.91 &1.92\\
\hline
\end{tabular}
\end{center}
\caption{Same information as in Table {\rm\ref{Ex1:1}} but for $\Delta t=h^{2}$} \label{Ex1:3}
\end{table}

\vskip5pt

\noindent{\bf Example 2.}
We next set the exact solution to the steady-state solution \cite{Chen2011}
$$
\begin{aligned}
{\bf u}_f({\bf x},t)&=\Big(\frac{1}{\pi}\sin(2\pi y)\cos x\,\,,\,\,\Big[2+\frac{1}{\pi^2}\sin^2(\pi y)\Big]\sin x\Big)\\
p_f({\bf x},t)&=0\\
\phi({\bf x},t)&=(e^{-y}-e^y)\sin x .
\end{aligned}
$$
It is easy to see, after scrutinizing the error analysis, that third order  instead of second order in time convergence is expected for $\phi$ and ${\bf u}_f$ in this steady-state case. Our numerical results in Table \ref{Ex2}, for which we have set $\Delta t= h$, suggest 3.5-order convergence for these variables. The convergence rates are consistent with the results of \cite{Chen2011} for the steady-state case. Thus, it seems that the super-convergence results of \cite{Chen2010} hold for our new temporal discretization schemes. The rigorous demonstration of the super-convergence rate can be accomplished by following the analyses of \cite{Chen2010} and will be discussed in future work.

\begin{table}[h!]
\begin{center}
\begin{tabular}{|c||c|c|c|c|c|c|}
\cline{2-7}
\multicolumn{1}{c}{} & \multicolumn{2}{|c|}{$e_\phi$} & \multicolumn{2}{c|}{$e_{\bf u}$} & \multicolumn{2}{c|}{$e_p$} \\
\hline
$h$ &  BDF2&  AMB2&  BDF2&   AMB2&  BDF2&   AMB2\\
\hline 1/16  &1.86e-005  &  2.70e-005  &  1.89e-005  &  1.36e-004 & 7.25e-003 &1.27e-001\\
\hline 1/32  &1.71e-006  &  2.17e-006  &  1.67e-006  &  1.02e-005 & 1.88e-003 &2.59e-002\\
\hline 1/64  &1.55e-007  &  1.79e-007  &  1.46e-007  &  7.39e-007 & 4.78e-004 &5.82e-003\\
\hline 1/128 &1.38e-008  &  1.51e-008  &  1.28e-008  &  5.47e-008 & 1.21e-004 &1.38e-003\\
\hline
\hline $r_{avg}$ & 3.46 & 3.60 & 3.51 &3.76 & 1.97 &2.17\\
\hline
\end{tabular}
\end{center}
\caption{Relative error and order of accuracy with respect to the spatial grid size $h$ for Example {\em2} at $t=1$ and with $\Delta t=h$.} \label{Ex2}
\end{table}

\vskip5pt
\noindent{\bf Example 3.}
To illustrate the long-time behavior of our schemes, we use the following exact solution that is a slight modification of an example in \cite{Cao2010b}:
$$
\begin{aligned}
{\bf u}_f({\bf x},t)&=\Big([x^2y^2+e^{-y}],[-\frac{2}{3}xy^3+[2-\pi\sin(\pi x)]]\Big)[2+\cos(2\pi t)] \\
p_f({\bf x},t)&=-[2-\pi\sin(\pi x)]\cos(2\pi y)[2+\cos(2\pi t)] \\
\phi({\bf x},t)&= [2-\pi\sin(\pi x)][-y+\cos(\pi(1-y))][2+\cos(2\pi t)] .\end{aligned}
$$
In this long time numerical experiment, we set the terminal time $T=100$,  and $h=1/64$. We choose $\Delta t=\frac{1}{128},\frac{1}{256}$ for BDF2 and $\Delta t=\frac{1}{256},\frac{1}{512}$ for AMB2. The relative errors are  plotted in Figures (\ref{fig1})--(\ref{fig3}). It is clear that although the errors grow initially, they remain bounded for all time. Moreover, the second-order in time accuracy for the velocity and the hydraulic head are also evident even in this onerous long-time experiment. The long-time accuracy for the pressure seems to be first-order in time, in agreement with our uniform in time error estimates. However, this is in contrast to the short-time second-order in time accuracy for $p$ as recorded in Table \ref{ShortTimeEx3}.

\begin{figure}[h!]
\centerline{
\includegraphics[height=1.75in]{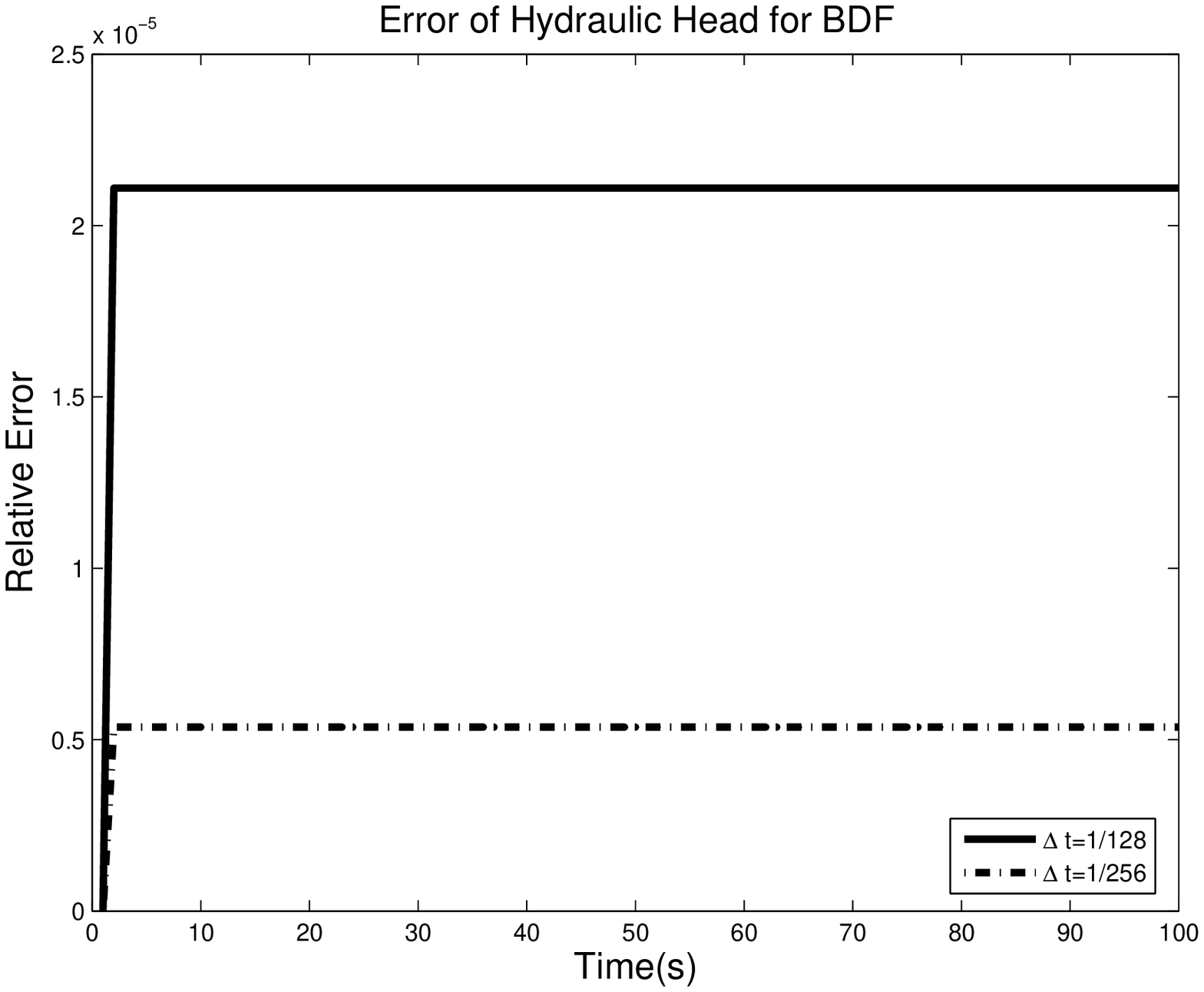}
\includegraphics[height=1.75in]{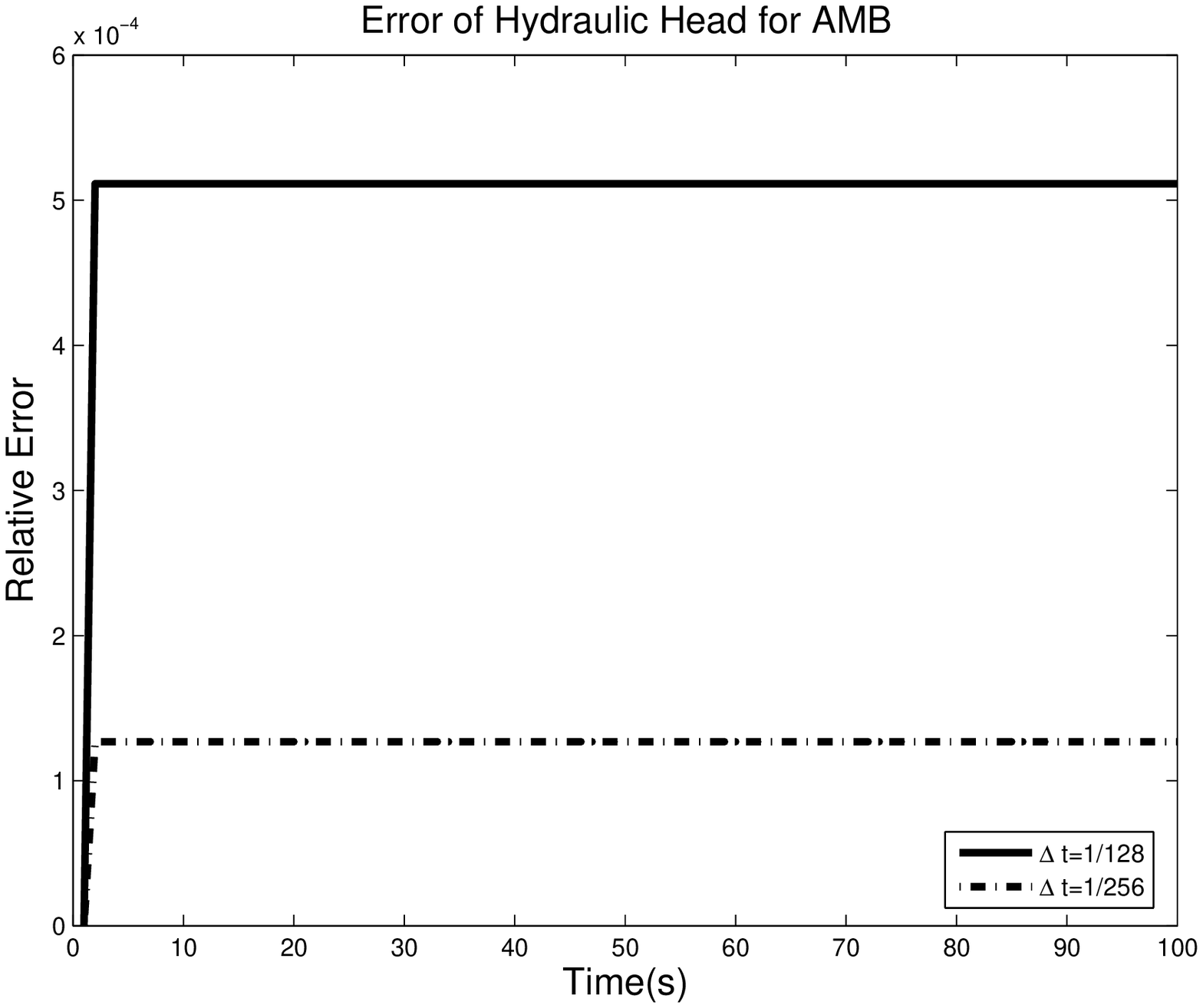}
}
\caption{Relative error for $\phi$ in Example {\em3} for {\em BDF2} (left) and {\em AMB2} (right) up to $t=100$ for $h=1/64$.}\label{fig1}
\end{figure}

\begin{figure}[h!]
\centerline{
\includegraphics[height=1.75in]{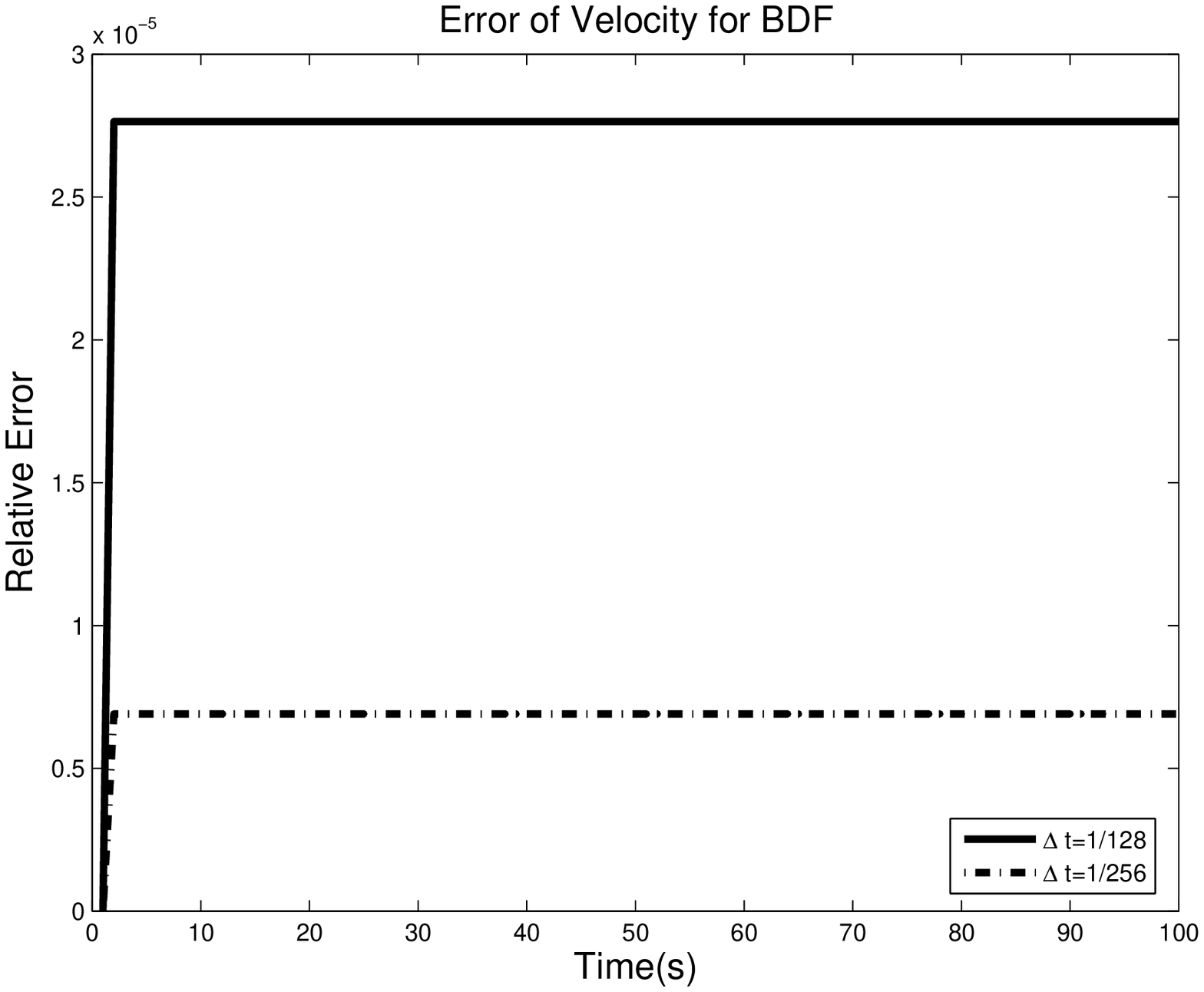}
\includegraphics[height=1.75in]{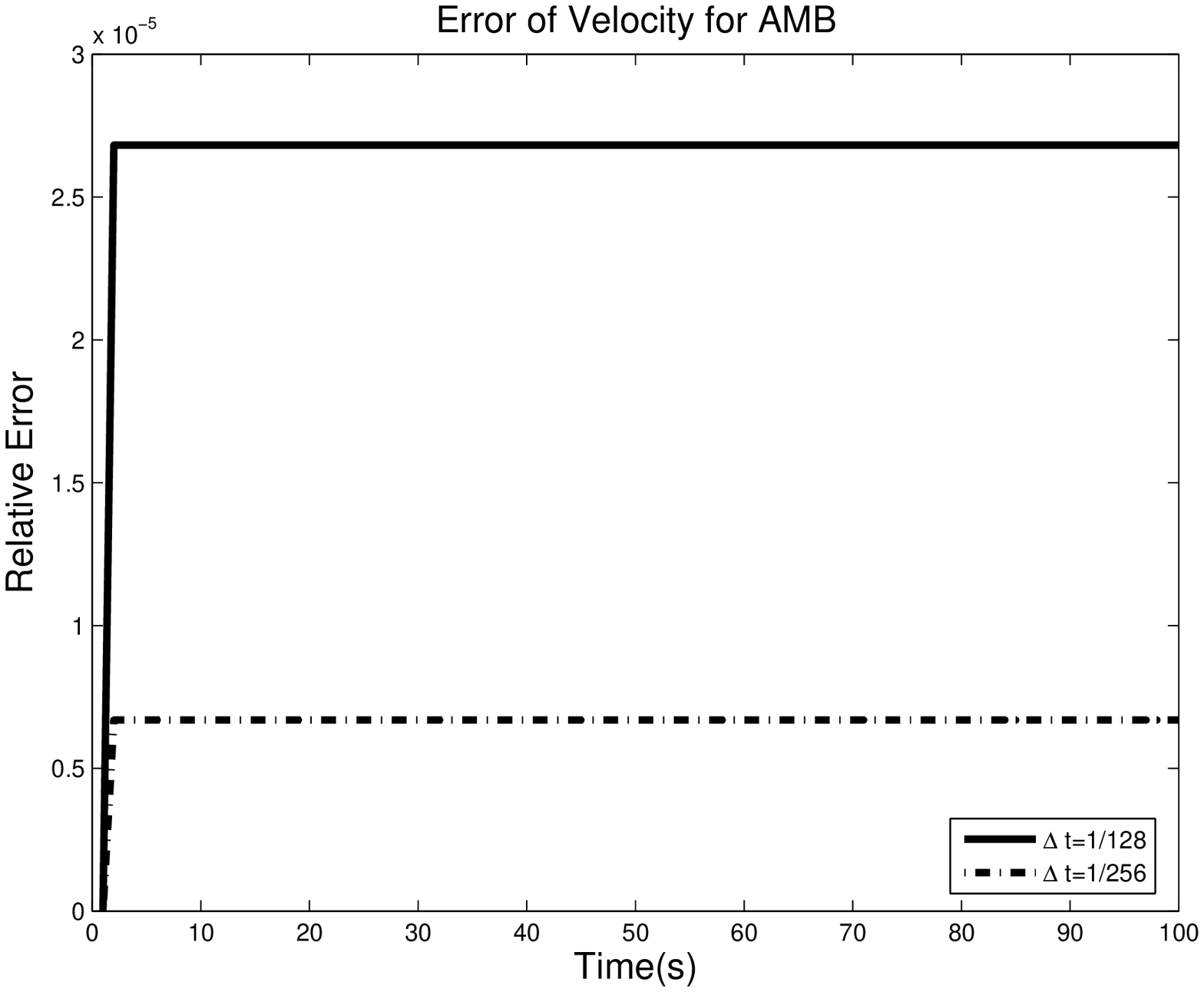}
}
\caption{Same information as for Figure {\em\ref{fig1}} but for ${\bf u}_f$.}\label{fig2}
\end{figure}

\begin{figure}[h!]
\centerline{
\includegraphics[height=1.75in]{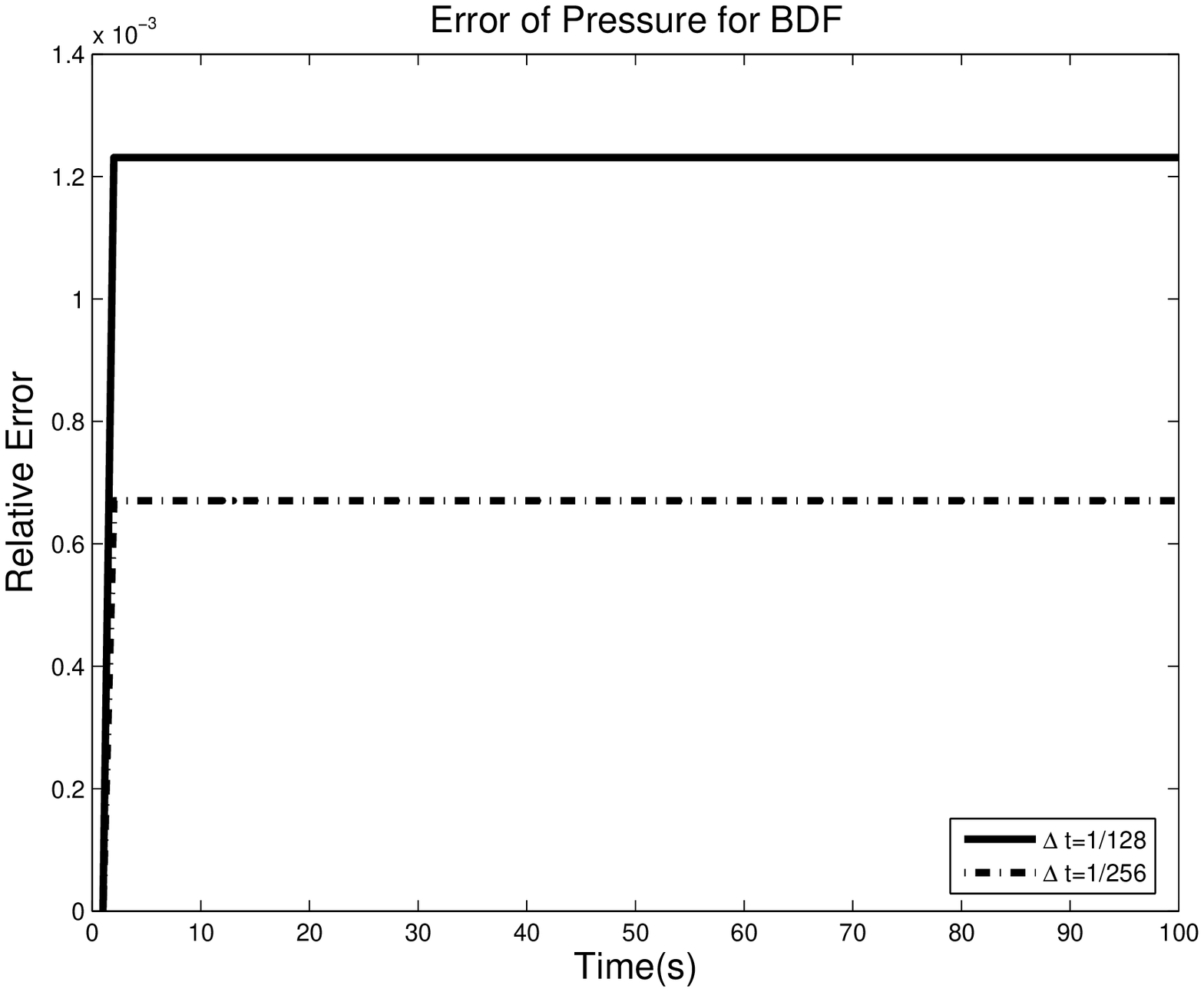}
\includegraphics[height=1.75in]{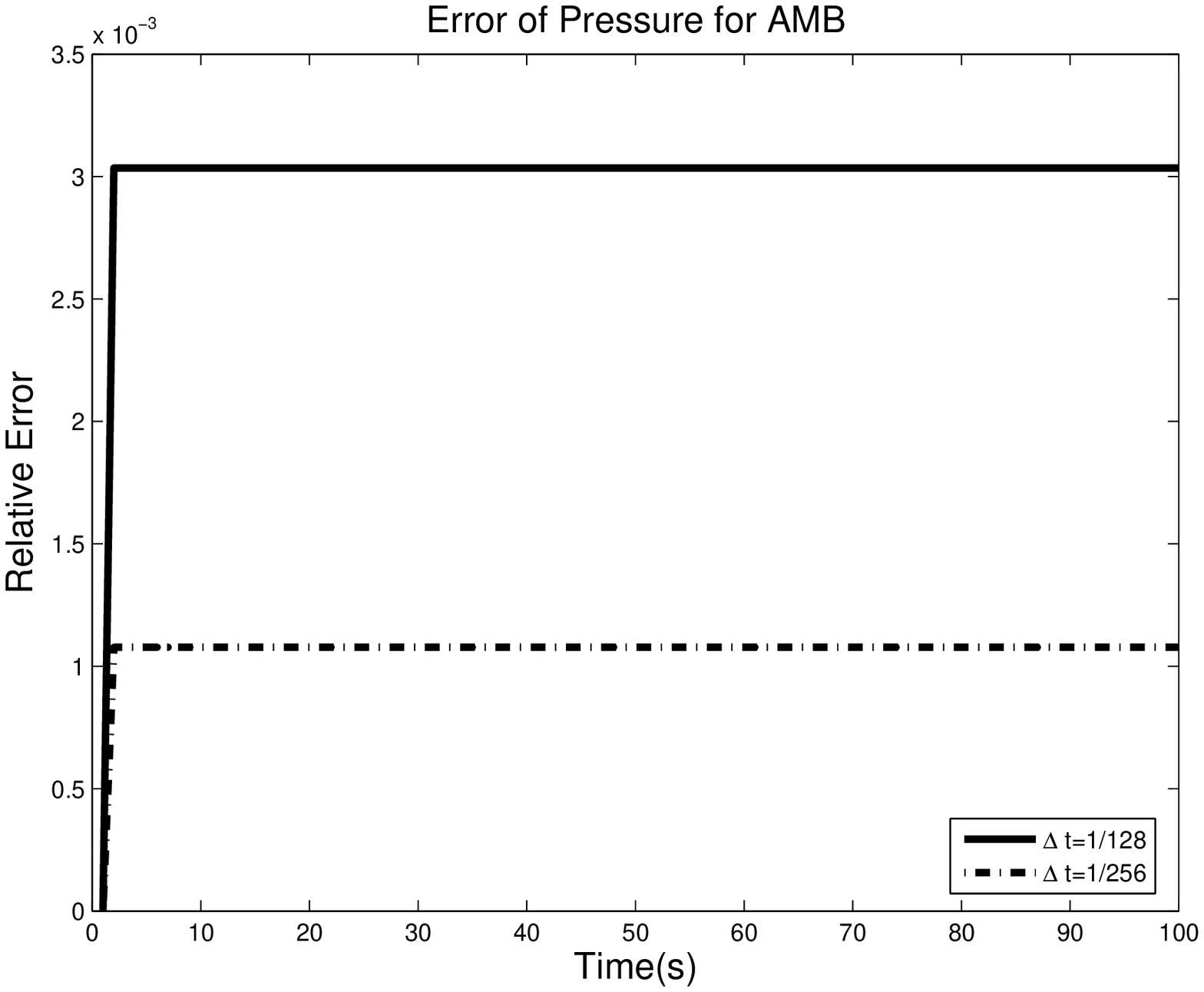}
}
\caption{Same information as for Figure {\em\ref{fig1}} but for $p$.}\label{fig3}
\end{figure}

\begin{table}[h!]
\begin{center}
\begin{tabular}{|c||c|c|c|c|c|c|}
\cline{2-7}
\multicolumn{1}{c}{} & \multicolumn{2}{|c|}{$e_\phi$} & \multicolumn{2}{c|}{$e_{\bf u}$} & \multicolumn{2}{c|}{$e_p$} \\
\hline
$h$ &  BDF2&  AMB2&  BDF2&   AMB2&  BDF2&   AMB2\\
\hline 1/16  & 2.05e-003 & 2.95e-002 &1.49e-003 &1.72e-003 &4.88e-002  &1.70e-001 \\
\hline 1/32  & 4.36e-004 & 7.76e-003 &4.18e-004 &4.26e-004 &1.40e-002  &4.32e-002 \\
\hline 1/64  & 9.84e-005 & 1.99e-003 &1.09e-004 &1.07e-004 &3.64e-003  &1.10e-002 \\
\hline 1/128 & 2.32e-005 & 5.05e-004 &2.75e-005 &2.68e-005 &9.29e-004  &2.79e-003 \\
\hline
\hline
$r_{avg}$ & 2.15 & 1.96 &1.92 &2.00 &1.91 &1.98  \\
\hline
\end{tabular}
\end{center}
\caption{Relative error and order of accuracy with respect to the spatial grid size $h$ for Example {\em3} at $t=1$ and with $\Delta t=h$.} \label{ShortTimeEx3}
\end{table}

\section{Concluding remarks}\label{sec:7}

We proposed and investigated two long-time accurate and efficient numerical methods for coupled Stokes-Darcy systems. The first is a combination of the second-order backward differentiation formula and the second-order Gear extrapolation method. The second is a combination of the second-order  Adams-Moulton and Adams-Bashforth methods. Our algorithms are special cases of the implicit-explicit (IMEX) schemes. The interfacial term that requires communication between the porous media and conduit, i.e., between the Stokes and Darcy components of the model, is treated explicitly in our algorithms so that only two decoupled problems (one Stokes and one Darcy) are solved at each time step. Therefore these schemes can be implemented very efficiently and, in particular, legacy codes can be used for each component.

We have shown that our schemes are unconditionally stable and long-time stable in the sense that solutions remain uniformly bounded in time. The uniform bound in time of the solution leads to uniform in time error estimates. This is a highly desirable feature because the physically interesting phenomena of contaminant sequestration and release usually occur over a very long time scale and one would like to have faithful numerical results over such time scales. Spatial discretization is effected using standard finite element methods. Time-uniform error estimates for the Darcy hydraulic head and the Stokes velocity and pressure for the fully-discrete schemes are also presented. These estimates are illustrated by numerical examples. The methods proposed can be also utilized to approximate steady-state solutions in case the problem data are time independent. All these features suggest that the two methods have strong potential in real applications.

On the other hand, there is still room for improvement. One could design even higher-order numerical methods. A third-order method was proposed in \cite{Cao2012} without analysis. We are currently developing third-order unconditionally stable schemes based on the Adams-Moulton-Bashforth approach. It is also desirable to use different and adaptive time-steps for the two regions involved due to the disparate time-scales in the two regions that one sees in practical situations; see, e.g., \cite{Li2011a,Li2011b}. Also, mortar element method can be naturally adopted and may be useful to efficiently handle the different spatial scales in the two subdomains; see, e.g., \cite{Layton2003}.

So far, all methods deal with confined (saturated) karst aquifers. Most  aquifers are unconfined and hence different methodologies involving either two-phase flows or free boundaries must be considered. Models for unconfined karst aquifers are inherently nonlinear. Mathematical investigation of unconfined karst aquifers is still in its infancy and deserves much needed attention.

Last but not the least, the application of these methods to the quantification of uncertainty in flow and contaminant transport would be of great interest in real applications that feature uncertainty in both the conduit geometry and matrix hydraulic conductivity.

\section*{Acknowledgements}
This work is supported in part by the National Science Foundation through DMS10008852, a planning grant from the Florida State University,
the Ministry of Education of China and the State Administration of Foreign Experts Affairs of China under a 111 project grant (B08018),
and the Natural Science Foundation of China under grant 11171077.

\end{document}